\documentclass[aap]{imsart}
\usepackage{amsthm,amsmath, amsfonts, amssymb}
\usepackage{cite}
\RequirePackage[colorlinks,citecolor=blue,urlcolor=blue]{hyperref}

\startlocaldefs

\numberwithin{equation}{section}
\theoremstyle{plain}
\newtheorem{lemma}{Lemma}[section]
\newtheorem{thm}{Theorem}[section]
\newtheorem{defn}{Definition}[section]
\newtheorem{prop}{Proposition}[section]

\newtheorem{remark}{Remark}[section]

\DeclareMathOperator*{\argmax}{arg\,max}

\newcommand{\cA}{\mathcal{A}}

\newcommand{\cC}{\mathcal{C}}
 \newcommand{\cD}{\mathcal{D}}

\newcommand{\cJ}{\mathcal{J}}

\newcommand{\cU}{\mathcal{U}}
\newcommand{\cV}{\mathcal{V}}
\newcommand{\cW}{\mathcal{W}}

\newcommand{\cY}{\mathcal{Y}}

\newcommand{\Del}{{\Delta}}

\newcommand{\pa}{{\partial}}

\newcommand{\ignore}[1]{}

\def \a{\alpha}
\def \b{\beta}

\def \d{\delta}
\def \1{\mathbf 1}

\def \E{\mathbb{E}}

\def \N{\mathbb{N}}

\def \P{\mathbb{P}}

\def \R{\mathbb{R}}

\newcommand{\comm}[1]{}
\DeclareMathOperator{\arctanh}{arctanh}
\DeclareMathOperator{\csch}{csch}
\makeatletter
\def\@setcopyright{}
\def\serieslogo@{}
\makeatother

\endlocaldefs

\begin{document}

\begin{frontmatter}

\title{On the asymptotic optimality of the comb strategy for prediction with expert advice}

 \author{\fnms{Erhan} \snm{Bayraktar}\corref{}\ead[label=e1]{erhan@umich.edu}\thanksref{ t1}}
  \thankstext{t1}{Supported in part by the National Science Foundation (NSF) under grant DMS-1613170 and the Susan M. Smith Professorship.} 
\address{erhan@umich.edu\\Department of Mathematics\\ 530 Church Street, Ann Arbor MI 48109}
 \affiliation{University of Michigan}

\and 

\author{\fnms{Ibrahim} \snm{Ekren}\ead[label=e2]{iekren@fsu.edu}}
 \address{iekren@fsu.edu\\ Department of Mathematics\\ 1017 Academic Way Tallahassee, FL 32306}
 \affiliation{Florida State University}
 \and 
 \author{\fnms{Yili} \snm{Zhang}\ead[label=e1]{zhyili@umich.edu}}
 \address{zhyili@umich.edu \\ Department of Mathematics\\ 530 Church Street, Ann Arbor MI 48109}
 \affiliation{University of Michigan}


\begin{abstract}
For the problem of prediction with expert advice in the adversarial setting with geometric stopping, we compute the exact leading order expansion for the long time behavior of the value function. Then, we use this expansion to prove that as conjectured in Gravin, Peres and Sivan \cite{MR3478415}, the comb strategies are indeed asymptotically optimal for the adversary in the case of $4$ experts.
\end{abstract}



\end{frontmatter}

\section{Introduction} 
In this paper we use PDE tools to analyze one of the classical problems in machine learning, namely prediction with expert advice. In this framework, a game is played between a player and the nature (also called the adversary in the learning literature). At each time step, given past information, the player has to choose an expert among $N>0$ experts. 
Simultaneously the nature chooses a set of winning experts. Then, both choices are announced. If the player chooses an expert belonging to the set of winning experts, the player also wins.
The objective of the player is to minimize his regret with respect to the best performing expert, i.e., minimize
$$R_T=\max_{i} G^i_T-G_T$$
where $G^i_T$ is the total gain of the expert $i$ and $G_T$ is the gain of the player at the final time. 
The objective of the nature is to choose the set of winning experts to maximize the regret of the player. This problem that has been extensively studied in learning theory \cite{covers,cesa-bianchi-lucosi,rakhlin-shanir-sritharan, MR3478415 , haussler, Littlestone, vovk} can also be seen as a discrete time and discrete space robust utility maximization problem similar to \cite{nutz} for a particular choice of utility function. 

For the case of 2 experts, the optimal strategy for the adversary was first described by Cover \cite{covers} in the 1960's. Recently, for the case with 3 experts, using an ansatz of exponential type for the value function of the game, Gravin et al. showed in \cite{MR3478415} that the so called comb strategy,  the strategy that consists of choosing the leading and the third leading expert by the nature, is optimal.  However, the exponential type ansatz for the value of the game does not generalize to larger number of experts.

In this paper, we follow the setting of \cite{MR3478415}, where the maturity of the game is a geometric random variable with parameter $\d>0$ and study the game where both the player and nature can use randomized strategies. In this framework, we prove 2 conjectures stated in \cite{MR3478415} for the game with $N=4$ experts. We use tools from stochastic analysis and PDE theory to give an explicit expansion of the value function of the game for small $\d>0$, which corresponds to long time asymptotics. In Theorem \ref{thm:main}, this expansion allows us to prove that the value of the game, also called best regret, indeed grows as  $\frac{\pi}{4\sqrt{2\d}}$ as conjectured in \cite{MR3478415}.

The proof of this result is achieved in two steps. This first step can be found in \cite{MR3768426}, where using tools from viscosity theory the author shows that the rescaled value function \eqref{eq:rescvalue} solves the elliptic PDE \eqref{eq:pdegeo}. The second step, which is the main contribution of this paper, is to explicitly solve this PDE for the case of 4 experts. In order to find this expression, we use the conjectured optimal strategy in \cite{MR3478415}, and relate the value function of the control problem \eqref{eq:fk} to an expectation of a functional of an obliquely reflected Brownian motion; see in particular Lemma~\ref{lem:reflecBM}. This expression is a discounted expected value of the local time that measures the number of times the best two experts' gains cross each other.  
Then, using appropriate differentiation of the dynamic programming equation \eqref{eq:expV}, we characterize the value of the expectation on two ``opposite" faces of the domain of reflection by a system of hyperbolic PDE \eqref{syst:hf} and \eqref{syst:hr}.
Then, we solve this system of hyperbolic PDE to explicitly compute the value for the conjectured control at the boundary, which then leads to the value in the whole domain. 
Finally,  in Section~\ref{s.proof}, we check that the value given for the conjectured control solve the nonlinear PDE \eqref{eq:pdegeo}, which is a simple verification argument proving the optimality of comb strategy, the second conjecture in \cite{MR3478415}. See Theorem~\ref{thm:asyopt}. The direct proof of the verification argument is quite tedious. We came up with a method that relies on Proposition \ref{prop:comp}, which is a type of maximum principle for the system of hyperbolic equations \eqref{eq:Fsuper}.   

From the perspective of control theory, we note that the setting of \cite{MR3478415} is in fact similar to the weak formulation (or feedback/closed loop formulation) of zero-sum games in the sense of \cite{pham-zhang} (see also \cite{MR3553921}) where the player and the nature observe the same source of information, i.e. the path of the gains of the experts and the player. One can also state the game in a Elliott-Kalton sense, see e.g. \cite{MR997385}, in which similarly to \cite{kohnserfaty}, before taking its decision, the nature learns the choice of the player. These two formulations generally lead to different values; see Remark 4.2 in \cite{MR3553921}.

Our expansion is in accordance with well known results in prediction problems. Indeed, it is known that in the long run, there is an upper bound for the value of regret minimization problems that grows at most as $\sqrt{\frac{T\log(N)}{2}}$ which is achieved by the so-called multiplicative weight algorithms \cite{cesa-bianchi-lucosi}. In this paper, we compute the exact scaling for the geometric stopping problem which also allows us to directly provide explicit algorithms for both the player and the nature.

The rest of the paper is organized as follows. In Section \ref{s.notation} we introduce our notation and define the value function of the regret minimization problem. In Section \ref{s.result}, we give the main results of the paper. This result is proven in Section \ref{s.proof}. The Sections \ref{s.oblique} and \ref{s.character} are there to provide the methodology used in finding the explicit solution \eqref{eq:u}. 

\section{Statement of the problem}\label{s.notation}

We fix $N\geq 2$ and denote by $U$ the set of probability measures on $\{1,\ldots, N\}$ and by $V$ the set of probability measures on $P(N)$, the power set of $\{1,\ldots N\}$. These sets of probability measures are in fact in bijection with respectively $N$ and $2^N$ dimensional unit simplexes. We denote by $\{e_i\}_{i=\{1,\ldots, N\}}$ the canonical basis of $\R^N$ and for $J\in P(N)$, $e_J$ stands for 
$e_J:=\sum_{j\in J}e_j$. Similarly to \cite{MR3055258}, for all $x\in \R^N$, we denote by $\{x^{(i)}\}_{i=1,\ldots, N}$ the ranked coordinates of $x$ with 
$$x^{(1)}\leq x^{(2)}\leq \ldots\leq x^{(N)},$$ and define the function 
\begin{align}\label{def:phi}
\Phi:x\in\R^N\mapsto \max_{i}x_i=x^{(N)}.
\end{align}

We assume that a player and the nature interact through the evolution of the state of $N$ experts. 
 At time $t\in \N$, the state of the game in hand is described by $\{G^i_s\}_{s=1,\ldots,t-1}$, the history of the gains of each expert $i=1,\ldots N$ and $\{G_s\}_{s=1,\ldots, t-1}$ the history of the gains of the player. 
At time step $t\in \N$, observing $\{(G^{i}_s,G_s):s=0,\ldots, t-1\}$, simultaneously, the player chooses $I_t\in \{1,\ldots,N\}$ and the nature chooses $J_t\in P(N)$. The gain of each expert chosen by the nature increases by $1$ i.e.,
\begin{align*}
G^i_{t}&=G^{i}_{t-1}+1\mbox{ if }i\in J_t \\
G^i_{t}&=G^{i}_{t-1}\mbox{ if }i\notin J_t .
\end{align*}
If the player also chooses an expert chosen by the nature, then the gain of player also increases i.e.,
\begin{align*}
G_{t}&=G_{t-1}+1\mbox{ if }I_t\in J_t \\
G_t&=G_{t-1}\mbox{ if }I_t\notin J_t .
\end{align*}
The regret of the player at time $t\in \N$ is defined as 
\begin{align*}
R_t:=\max_{i=1,\ldots,N}G^i_t-G_t.
\end{align*}
Let $T$ denote the random maturity of the problem. We assume that $T$ is a geometric random variable with parameter $\d>0$. 

We now convexify the problem by assuming that instead of choosing deterministic $J_t$ and $I_t$, the nature and the player choose randomized strategies. At time $t$, the player chooses a probability distribution $\a_t\in U$ and the nature chooses $\b_t\in V$ that may depend on the observation $\{(G^{i}_s,G_s):s=0,\ldots t-1,\, i=1,\ldots ,N\}$. We denote by $\cU$ the set of such sequences $\{\a_t\}_{t\in \N}$ and by $\cV$ the set of such sequences $\{\b_t\}$. With some notational abuse, we denote by $I_t\in \{1,\ldots N\}$ the random variable with distribution $\a_t$ and $J_t\in P(N)$ the random variable with distribution $\b_t$.  

The objective of the player is to minimize his expected regret at time $T$ and the objective of the nature is to maximize the regret of the player. Hence we have a zero sum game with the lower and the upper value for the game
\begin{align*}
\sup_{\b\in\cV}\inf_{\a\in \cU} \E^{\a,\b}\left[R_T\right]\leq\inf_{\a\in \cU}\sup_{\b\in\cV}  \E^{\a,\b}\left[R_T\right]
\end{align*}
where $\E^{\a,\b}$ is the probability distribution under which we evaluate the regret given the controls $\a=\{\a_t\}$ and $\b=\{\b_t\}$. We denote by
\begin{align}\label{eq:X}
X_t:=(X^1_t,\ldots, X_t^N):=(G^1_t-G_t,\ldots, G^N_t-G_t)
\end{align}
the difference between the gain of the player and the experts. The following result, which can be found in \cite{MR3478415,MR3768426}, establishes the existence of a value for this discrete game.  \begin{prop}
The game has a value, i.e.,
\begin{align}\label{eq:defVd}
V^\d(X_0):=\sup_{\b\in\cV}\inf_{\a\in \cU} \E^{\a,\b}\left[R_T\right]=\inf_{\a\in \cU}\sup_{\b\in\cV}  \E^{\a,\b}\left[R_T\right].
\end{align}
There exists $M>0$ independent of $\d$ such that for all $\d>0$ and $x\in \R^N$ we have that
\begin{align*}
|V^\d(x)-\Phi(x)|\leq \frac{M}{\sqrt{\d}}.
\end{align*}
Additionally, $V^\d$ satisfies the following dynamic programming principle
\begin{align*}
& V^\d(x)=\delta \Phi(x)+(1-\d)\inf_{\a\in U} \sup_{\beta\in V}\sum_{J} \beta_J \left( V^\d\left(x+e_J\right)-\a(J)\right).
\end{align*}

\end{prop}
\begin{proof}
The existence of the value is a direct consequence of the Minimax Theorem and is provided in \cite{MR3478415}. The proof of the rest of the Proposition can be found in \cite{MR3768426}. In particular the uniform bound in $x$ is a consequence of \cite[Theorem 3]{MR3768426}.
\end{proof}

\subsection{Limiting behavior of $V^\d$ }
The main objective of the paper is to provide an explicit formula for the leading order for the function $V^\d$ for small $\d>0$. For this purpose define the rescaled value function:
\begin{align}\label{eq:rescvalue}
u^\d:x\in \R^N\mapsto V^\d\left(\frac{x}{\sqrt{\d}}\right)\sqrt{\d}.
\end{align}
The next result shows that the limiting behavior of the value of the game can be characterized by the value of a stochastic control problem. 
\begin{prop}\label{prop:dppud} 
As $\d\downarrow 0$, the function $u^\d$ converges locally uniformly to $u:\R^N\mapsto \R$ which is the unique viscosity solution of the equation
\begin{align}\label{eq:pdegeo}
u(x)-\frac{1}{2}\sup_{J\in P(N)} e_J^\top \pa^2 u(x)e_J= \Phi(x)
\end{align}
in the class of functions with linear growth. 
Additionally, $u$ admits the stochastic control representation
\begin{align}\label{eq:fk}
u(x)=\sup_{(\sigma_t)} \E\left[\int_0^\infty e^{-t}\Phi(X_t)dt\right]
\end{align}
where $X$ is defined by $X_t=x+\int_0^t\sigma_s dW_s$ with $W$ a 1-dimensional Brownian motion and the progressively measurable process $(\sigma_t)$ satisfies for all $t$ $\sigma_t\in \{e_J:J\in P(N)\}$.
\end{prop}
\begin{proof}
The fact that $u^\d$ converges to $u$ is a consequence of \cite[Theorem 7]{MR3768426}. Note also that an analysis of the proof of \cite[Theorem 7]{MR3768426} and the general methodology of proof in \cite{MR1115933} allows us to claim that the convergence is in fact locally uniform. The fact that $u$ admits the representation \eqref{eq:fk} is a consequence of uniqueness of viscosity solution of \eqref{eq:pdegeo} with linear growth that is proven in \cite[Theorem 5.1]{MR1118699} and the stochastic Perron's method of \cite{MR3124891}.
\end{proof}
\section{Main Results}\label{s.result}
\subsection{Explicit solution for $4$ experts}
The main contribution of the paper is to provide a method to explicitly solve the PDE \eqref{eq:pdegeo}.
\begin{thm}\label{thm:main} With $4$ experts, for $x\in \R^4$, the function $u$ is given by the expression
\begin{align}\label{eq:u}
 u(x)&=x^{(4)}-\frac{\sqrt{2}}{4} \sinh(\sqrt{2}(x^{(4)}-x^{(3)}))\\
&+\frac{\sqrt{2}}{2} \arctan\left(e^{\frac{x^{(1)}+x^{(2)}-x^{(3)}-x^{(4)}}{\sqrt{2}}}\right)\cosh\left(\frac{x^{(1)}-x^{(2)}+x^{(3)}-x^{(4)}}{\sqrt{2}}\right)\notag\\
&\quad\cosh\left(\frac{-x^{(1)}+x^{(2)}+x^{(3)}-x^{(4)}}{\sqrt{2}}\right)\cosh\left(\frac{-x^{(1)}-x^{(2)}+x^{(3)}+x^{(4)}}{\sqrt{2}}\right)\notag\\
    &+\frac{\sqrt{2}}{2}\arctanh\left(e^{\frac{x^{(1)}+x^{(2)}-x^{(3)}-x^{(4)}}{\sqrt{2}}}\right)\sinh\left(\frac{x^{(1)}-x^{(2)}+x^{(3)}-x^{(4)}}{\sqrt{2}}\right)\notag\\
    &\quad\sinh\left(\frac{-x^{(1)}+x^{(2)}+x^{(3)}-x^{(4)}}{\sqrt{2}}\right)\sinh\left(\frac{-x^{(1)}-x^{(2)}+x^{(3)}+x^{(4)}}{\sqrt{2}}\right)\notag
\end{align}
Additionally, $u$ is twice continuously differentiable, monotone \footnote{Monotone here means
$$u(x_1+y_1,\dots,x_N+y_N)\geq u(x_1,\dots,x_N)\mbox{ for all }x_i\in \R\mbox{ and }y_i\geq 0.$$ }, symmetric in its variables on $\R^4$, satisfy 
\begin{align}\label{translate}
u(x+\lambda (e_1+e_2+e_3+e_4))=u(x)+\lambda\mbox{ for all }x\in\R^4\mbox{ and }\lambda \in \R
\end{align}
and if $J$ is a maximizer of the Hamiltonian $\sup_{J\in P(N)} e_J^\top \pa^2 u(x)e_J$ then its complement $J^c$ is also a maximizer of the same Hamiltonian. 

Moreover, 
\begin{equation}\label{cor:0}
V^{\delta}(0)=\frac{\pi}{4\sqrt{2\d}}+o\left(\frac{1}{\sqrt{\d}}\right),
\end{equation}
In fact, $u$  has the following expansion at the origin
\begin{equation}\label{eq:expofu}
\begin{split}
&u(x_1,x_2,x_3,x_4)=\frac{\pi}{4\sqrt{2}}+\frac{1}{4}(x_1+x_2+x_3+x_4)+\\
&\frac{3\pi}{16\sqrt{2}}(x_1^2+x^2_2+x^2_3+x^2_4-\frac{2}{3}(x_1x_2+x_1x_3+x_1x_4+x_2x_3+x_2x_4+x_3x_4)) \\
&+ o(|x|^2).
\end{split}
\end{equation}
\end{thm}
\begin{proof}
The proof of this result is provided in Section \ref{s.proof} after developing the methodology required to obtain this expression. Note that one can check by hand (or preferably with a computer) that the expression provided at \eqref{eq:u} solves the equation \eqref{eq:pdegeo} when all $x_i$ are different from each other. Since the set of points $x\in \R^4$ with $x_i=x_j$ for some $i,j$ is of zero Lebesgue measure, this proves that $u$ is an almost everywhere solution of \eqref{eq:pdegeo}. However, due to potential discontinuities of the derivatives when two of the $x_i's$ are equal we need to check that the almost everywhere solution of the equation \eqref{eq:pdegeo} defined via this expression is twice continuously differentiable and is therefore a smooth solution. 
\end{proof}

\begin{remark}
i)\eqref{cor:0} is the main result for the long time behavior of the regret minimization problem with geometric stopping and is conjectured in \cite{MR3478415}. The optimal regret scales as the square root of the time scale in hand. In this case of geometric stopping $u(0)=\frac{\pi}{4\sqrt{2}}$ gives the term of proportionality between the optimal regret and the stopping time parameter. 

ii) The fact that $J^c$ maximizes $\sup_{J\in P(N)} e_J^\top \pa^2 u(x)e_J$ whenever $J$ maximizes this Hamiltonian is a direct consequence of the regularity of $u$ and its translation invariance as in \eqref{translate}. This fact will be useful to us while checking the optimality of comb strategies.

 \end{remark}
\subsection{Asymptotically optimal strategies} Given the value of $u$, we now describe a family of asymptotically optimal strategies for nature. 
Inspired by \cite{MR3478415} we give the following definition. 
\begin{defn}\label{def:controls}
(i) We denote 
\begin{align}\label{def:J*}
\cJ^*(x)=\argmax_{J\in P(N)} e_{J}^\top \pa^2 u(x)e_{J},
\end{align}
the set of maximizers of the Hamiltonian.

(ii) For all $x\in \R^4$ with $x_{i_1}\leq x_{i_2}\leq x_{i_3}\leq x_{i_4}$, we denote $\cJ_\cC(x)\in P(N)$ the comb strategy which is the control for the problem \eqref{eq:fk} that consists in choosing the experts $i_4$ and $i_2$. We take the convention that if two components $x_i$ and $x_j$ of the points are equal for $i<j$ then the ordering of the point is taken with $x_i\leq x_j$.
  
 (iii) We denote $\cJ^b_\cC\in \cV$ the balanced comb strategy which is the control for the nature in game \eqref{eq:defVd} that consists in choosing at $\frac{x}{\sqrt{\d}}\in \R^4$, $\cJ_\cC(x)\in P(N)$ with probability $\frac{1}{2}$ and $\cJ^c_\cC(x)\in P(N)$ with probability $\frac{1}{2}$.
\end{defn}
\begin{remark}
Note that (ii) defines a control for the control problem \eqref{eq:fk} while (iii) defines a control for the game \eqref{eq:defVd}. Hence the latter depends on $\d$ and $x$ and is scaled to reflect the scaling between the two problems. Additionally, as a consequence of \cite[Claim 1]{MR3478415}, we have defined $\cJ^b_\cC$ as the unique balanced strategy that can be generated using $\cJ_\cC(x)$.
\end{remark}
One may conjecture that it is asymptotically optimal for the nature to choose for all $\frac{x}{\sqrt{\d}}\in \R^4$ an element in $\cJ(x)$. However, this conjecture is not true since the strategy is not balanced in the sense of \cite{MR3478415}. Indeed, assume for example that for $x\in \R^4$ $\cJ^*(x)$ is reduced to a unique subset of cardinality $1$, meaning $\cJ^*(x)=\{J\}=\{\{i\}\}$. In this case, choosing the expert $i$ would be suboptimal for the nature since the player can also guess this control and choose the expert $i$. It is proven in \cite{MR3478415} that in order to be optimal any strategy of the nature has to be balanced. Thanks to the Theorem \ref{thm:main}, the simplest strategy for the nature would be to randomize his strategy between the maximizer of the Hamiltonian and its complement.

The main result for asymptotically optimal strategies is the following theorem.  
\begin{thm}\label{thm:asyopt}
The control $\cJ^b_\cC\in \cV$ is asymptotically optimal for the nature, in the sense that 
\begin{align}
\underline u^\d(x)=u(x)+o(1),
\end{align}
where $o$ is locally uniform in $x$, and 
we denote
\begin{align}
\underline u^\d(x)=\sqrt{\d}\inf_{\a\in \cU}\E^{\a,\cJ^b_\cC}\left[{R^{\frac{x}{\sqrt{\d}}}_T}\right]\label{def:uu}\end{align}
where $R^{\frac{x}{\sqrt{\d}}}_T$ is the regret of player at time $T$ starting from the state $X_0=\frac{x}{\sqrt{\d}}$.
\end{thm}

The proof is deferred to Section~\ref{sec:profossrst}. We will finish this section with a few remarks.

\begin{remark}
As a sanity check, the expansion of $u$ implies that the Hessian of $u$ is 
$$H=\frac{\pi}{8\sqrt{2}}\begin{pmatrix}
3&-1&-1&-1\\
-1&3&-1&-1\\
-1&-1&3&-1\\
-1&-1&-1&3
\end{pmatrix}$$
and 
$$u(0)=\frac{\pi}{4\sqrt{2}}=\frac{1}{2}\begin{pmatrix}
1\\
0\\
1\\
0
\end{pmatrix}^\top H\begin{pmatrix}
1\\
0\\
1\\
0
\end{pmatrix},$$
where the second equality follows from \eqref{eq:pdegeo} and the optimality of the comb strategies.
\end{remark}

\begin{remark}
We note that at the leading order it is optimal for the nature to choose the controls $\cJ_\cC^b$ in the sense that for all family 
$\a^\d\in \cU$ and $\beta^\d \in \cV$ for $\d>0$,  we have that 
$$\limsup_{\delta\downarrow 0}\sqrt{\delta}\left( \E^{\a^\d,\b^\d}\left[{R^{\frac{x}{\sqrt{\d}}}_T}\right]-\E^{\a^\d,\cJ^b_\cC}\left[{R^{\frac{x}{\sqrt{\d}}}_T}\right]\right)\leq 0.$$
This inequality means that up to an error negligible at the leading order, the comb strategy is optimal for the nature. 
 \end{remark}

\begin{remark}
i)In the case of 3 experts, \cite{MR3478415} gives the exact value of $V^\d$ based on a ``guess and verify approach".The following expression is given for $u$ in \cite{MR3768426}
$$x^{(3)}+\frac{1}{2\sqrt{2}}e^{\sqrt{2}(x^{(2)}-x^{(3)})}+\frac{1}{6\sqrt{2}}e^{\sqrt{2}(2x^{(1)}-x^{(2)}-x^{(3)})},$$
which is obtained by taking a continuum analogue of \cite{MR3768426}.
Compared to this 3 dimensional counterpart the expression \eqref{eq:u} is not a simple sum of exponentials. Instead of guess and verify we needed  to directly compute the value of comb strategies.
 
 ii)It is possible to prove the Theorem \ref{thm:asyopt} for the particular case of $x=0$ using a combination of Theorem \ref{thm:main} and \cite[Theorem 5.1]{MR3478415}. However, unlike the proof provided below, such a proof cannot be directly extended to the general case $x\in \R^4$. 
 \end{remark}

\begin{remark}
Note that for all $x\in \R^4$ we have $\pa_i u(x)\in [0,1]$ with $\sum_{i=1}^4\pa_i u(x)\in [0,1]=1$. Hence $\{\pa_i u(x)\}_{i=1}^4\in U$. The claim is direct consequence of \eqref{translate}.
Thanks to this observation, we can define $\a^*\in \cU$ via the feedback control :
at point $\frac{x}{\sqrt{\d}}\in \R^4$, the player chooses the expert $i$ with probability $\pa_i u(x)$ and define the value 
\begin{align}
\overline u^\d(x)=\sqrt{\d}\sup_{\b\in \cV}\E^{\a^*,\b}\left[{R^{\frac{x}{\sqrt{\d}}}_T}\right]\label{def:ou}.
\end{align}
We conjecture that
$$\overline u^\d(x)=u(x)+o(1)$$
which would imply that $\a^*$ is an asymptotically optimal strategy for the player. The main difficulty one faces to obtain such a result is to obtain locally uniform bounds for $\overline u^\d(x)$ when $\d\downarrow 0$. 
\end{remark}

 \section{Value for comb strategies}\label{s.oblique}
Inspired by the conjecture in \cite{MR3478415}, our objective here is to introduce the value of the control problem \eqref{eq:fk} corresponding to comb strategies. Then, in Section \ref{s.character}, we develop a methodology to compute this value.  Finally, in Section \ref{s.proof}, we check that the value computed in these sections is a solution to \eqref{eq:pdegeo}. 

We note that the Sections \ref{s.oblique} and \ref{s.character} are only included in the paper to explain how to find the expression \eqref{eq:u}. Indeed, the only rigorous proofs for our results are in Section \ref{s.proof}. Therefore, in Sections \ref{s.oblique} and \ref{s.character}, we will slightly deviate from mathematical rigor.  The purpose of this section is to relate the value given by comb strategies with distributional properties of an obliquely reflected Brownian Motion. Then, we compute and analyze this value in Sections \ref{s.character} and \ref{s.proof}.  

\subsection{Analysis}
The optimal strategy for \eqref{eq:fk} conjectured in \cite{MR3478415} consists in choosing the best and the third best experts. This is a rank based interaction for the evolution of the components of $X^\cC$, the optimally controlled state. Therefore, for any $x\in \R^4$, it is expected that $X^{i,x,\cC}$ solves the following SDE
 \begin{align}\label{eq:combcontrol}
 X_t^{i,x,\cC}=x^i+\int_0^t\sum_{j=1}^4 \sigma^\cC_j \1_{\{X_r^{i,x,\cC}=X_r^{(j),x,\cC}\}}dW_r\mbox{ for }t\geq 0\mbox{ and }i=1,\dots,4; 
 \end{align}
 where $\sigma^\cC_4=\sigma^\cC_2=1$ and $\sigma^\cC_3=\sigma^\cC_1=0$ is the control corresponding to comb strategy.

It is not clear that \eqref{eq:combcontrol} admits a strong solution. In fact, based on \cite[Theorem 4.1]{fernholz2011planar}, we conjecture that there is no strong solution to \eqref{eq:combcontrol}. However, it is expected that the ranked components $X_t^{(i),x,\cC}$ are well-defined. Given also the fact that the payoff of the problem is symmetric, we will directly define our value of interest via an obliquely reflected Brownian motion. This procedure also allows a reduction of the dimension of the problem.

We first recall the definition of an obliquely reflected Brownian motion given in \cite[Definition 2.1]{MR1381009}.
\begin{defn}\label{def:srbm}
We say that the family of continuous processes $\{\cY^{y}_t\}_{y\in \R_+^3}$ and probability measures $\{\P^{y}\}_{y\in \R^3_+}$ is a weak solution to the semimartingale reflected Brownian motion on $\R^3_+$ with covariance matrix $\Gamma$ and reflection matrix $R$ if \\
i) For all $t\geq 0$ and $y\in \R^3_+$
$$\cY^{y}_t=\cW_t^y+R\lambda^y(t)$$
ii) The process $\cW_t^y\in\R^3$ is a Brownian motion with covariance matrix $\Gamma$ under $\P^y$.\\
iii) $\lambda^y$ is adapted to the filtration generated by $\cY^{y}$, $\lambda_0^y=0$, $\lambda^y$ is continuous, non decreasing, and 
$$\int_0^t\1_{\{\cY^{i,y}_r=0\}}d\lambda^i_r=\lambda^i_t\mbox{ for }i=1, 2, 3.$$
\end{defn}
We will denote by $(Y^{y})_{y \in \mathbb{R}_+^3}=(Y^{1,y},Y^{2,y},Y^{3,y})_{y \in \mathbb{R}_+^3}$ the family
with
$$\Gamma:=\begin{pmatrix}
1&-1&1\\
-1&1&-1\\
1&-1&1
\end{pmatrix},\,R:=\begin{pmatrix}
1&-1/2&0\\
-1/2&1&-1/2\\
0&-1/2&1
\end{pmatrix}$$
and $(Y^y_1(0),Y^y_2(0),Y^y_3(0))=y \in \R^3_+$.
These processes have the 
following semimartingale decomposition for $t\geq 0$,
\begin{align}\label{eq:decomY}
dY_t^{3,y}&=dW_t+d\Lambda^{3}_t-\frac{1}{2}d\Lambda^{2}_t,\notag \\
dY_t^{2,y}&=-dW_t+d\Lambda^2_t-\frac{1}{2}(d\Lambda^{3}_t+d\Lambda^{1}_t),\\
dY_t^{1,y}&=dW_t+d\Lambda^{1}_t-\frac{1}{2}d\Lambda^{2}_t,\notag
\end{align}
and denote 
$\Lambda^{j}_t$ for $j=1,2,3$ the local time of $Y^y_j\geq 0$  at the origin. Since the matrix $R-I$ is a tridiagonal Toeplitz matrix whose eigenvalues are less than $1$ in absolute value, there exists a unique solution to the oblique reflection problem;  see \cite[Theorem 2.1]{MR1381009}.
However, the existence of solution to \eqref{eq:combcontrol} is not straightforward as discussed above. If a solution to this system existed, then we clearly would have 
\begin{align*}
Y_t^{j-1,y}(t)= X_t^{(j),x,\cC}- X_t^{(j-1),x,\cC}\geq 0\mbox{ for } j=2, 3, 4,
\end{align*}
with $y=(x^{(2)}-x^{(1)},x^{(3)}-x^{(2)},x^{(4)}-x^{(3)})\in \R^3_+$. Henceforth, we will assume that this is the case. (This is the only non-rigorous part of the derivation. But we should again remark that a rigorous verification of our claims is in Section~\ref{s.proof} and the arguments here are performed for giving an intuitive construction of the solution.)
In the sequel we will denote
\[
Y^{4,y}_t=\sum_{j=1}^4X_t^{(j),x,\cC}=\sum_{j=1}^4X_t^{j,x,\cC}.
\]

\subsection{Value associated to an obliquely reflected Brownian motion}
We now give a lemma that allows us to define our candidate solution to \eqref{eq:pdegeo}.
\begin{lemma}\label{lem:reflecBM}
Assume that there exists a weak solution to \eqref{eq:combcontrol}. Then for all $x$ we have 
\begin{align}\label{eq:decomuc}
\E\left[\int_0^\infty e^{-t}\Phi(X^{x,\cC}_t)dt\right]=&\Phi(x)+v(x^{(2)}-x^{(1)},x^{(3)}-x^{(2)},x^{(4)}-x^{(3)})
\end{align}
where 
\begin{align*}
&v(y_1,y_2, y_{3}):=\frac{1}{2}\E\left[\int_0^\infty e^{-t}\Lambda^{3,y}_tdt\right].
\end{align*}

\end{lemma}
\begin{proof}
We fix $x\in \R^d$ define $y\in \R_+^3$ with $y:=( x^{(2)}-x^{(1)},x^{(3)}-x^{(2)},x^{(4)}-x^{(3)})$. 
Thanks to our definitions for all $t\geq0$, 
\begin{align*}
 Y^{1,y}_t+2Y^{2,y}_t+3Y^{3,y}_t&=3 X^{(4),x,\cC}_t-X^{(3),x,\cC}_t-X^{(2),x,\cC}_t-X^{(1),x,\cC}_t\\
 &=4 X^{(4),x,\cC}_t-Y^{4,y}_t.
 \end{align*}
Thus, 
\begin{align}\label{eq:aux1}
\E\left[\int_0^\infty e^{-t}\Phi(X^{x,\cC}_t)dt\right]=\frac{1}{4}\E\left[\int_0^\infty e^{-t} Y^{4,y}_tdt\right]+\frac{1}{4}\E\left[\int_0^\infty e^{-t}\sum_{k=1}^{3}k Y^{k,y}_tdt\right].
\end{align}
Note that $Y^{4,y}$ is a martingale, and by differentiation and \eqref{eq:decomY}
 $$d\left(\sum_{k=1}^{3}k Y^{k,y}_t\right)=dW_t+2d\Lambda^{3,y}_t.$$
 Therefore, 
\begin{align}\label{eq:aux2}
\E\left[\int_0^\infty e^{-t}\Phi\left(X^{x,\cC}_t\right)dt\right]=\sum_{i=1}^4 \frac{x_i}{4}+\frac{1}{4}\sum_{i=1}^3iy_i+\frac{1}{2}\E\left[\int_0^\infty e^{-t}\Lambda^{3,y}_tdt\right].
\end{align}
Thus, by the definition of $v$, and a simple algebraic verification for the first two terms on the right, we have the equality \eqref{eq:decomuc}. 
\end{proof}
\begin{remark}
One interpretation of the previous lemma is that the optimal strategy aims to maximize the third component of the local time of a reflected Brownian motion. This is consistent with discrete time problem in the case $N=2$ or $N=3$ where the optimal strategies of the nature is proven to be maximizer of the number of crossings between the leading and the second leading experts \cite{covers,MR3478415}. We note that this strategy also maximize the expected value of $\sum_{k=1}^{3}kY^{k,y}_{\tau}$  where $\tau$ is exponentially distributed. 
\end{remark}

\begin{prop}
The function defined by 
\begin{align}\label{eq:defv}
&v:y\in\R^3_+\mapsto \frac{1}{2}\E\left[\int_0^\infty e^{-t}\Lambda^{3,y}_tdt\right]
\end{align}
 is a viscosity solution of 
\begin{align}\label{eq:pdev}
0=v-\frac{1}{2}\begin{pmatrix}
1\\
-1\\
1
\end{pmatrix}^T \pa^2 v\begin{pmatrix}
1\\
-1\\
1
\end{pmatrix},\mbox{ on }(0,\infty)^3
\end{align}
with the reflection conditions
\begin{align}
\pa_{3} v-\frac{1}{2}\pa_{2}v&=-\frac{1}{2}\mbox{ if }y_{3}=0,\mbox{ and }(y_1,y_2)\in (0,\infty)^2,\label{eq:refl1}\\
\pa_{2} v-\frac{1}{2}(\pa_{1}v+\pa_3v)&=0\mbox{ if }y_{2}=0\mbox{ and }(y_1,y_3)\in (0,\infty)^2,\label{eq:refl2}\\
\pa_{1} v-\frac{1}{2}\pa_{2}v&=0\mbox{ if }y_{1}=0\mbox{ and }(y_2,y_3)\in (0,\infty)^2\label{eq:refl3}.
\end{align}

\end{prop}
\begin{proof}
We introduce the auxiliary function 
\begin{align}\label{eq:deftildev}
\tilde v(y)=\frac{1}{4}\sum_{i=1}^3i y_i + v(y)\mbox{ for all }y\in \R^3_+.
\end{align}
Thanks to \eqref{eq:aux1} and \eqref{eq:aux2}, 
$\tilde v(y)=\frac{1}{4}\E\left[\int_0^\infty e^{-t}\sum_{k=1}^{3}k Y^{k,y}_tdt\right]$.
For all stopping time $\tau\geq 0$, the dynamic programming principle leads to 
$$\tilde v(y_1,y_2,y_3)=\E\left[\int_0^\tau\frac{  e^{-t}}{4}\sum_{k=1}^{3}k Y^{k,y}_tdt +e^{-\tau}\tilde v(Y_\tau^{y})\right].$$
Using the martingality of $\cY^y$ on $(0,\infty)^3$, we obtain that on $(0,\infty)^3$, 
$$0=\tilde v-\frac{  1}{4}\sum_{k=1}^{3}k y_k-\frac{1}{2}\begin{pmatrix}
1\\
-1\\
1
\end{pmatrix}^T \pa^2 \tilde v\begin{pmatrix}
1\\
-1\\
1
\end{pmatrix}=v-\frac{1}{2}\begin{pmatrix}
1\\
-1\\
1
\end{pmatrix}^T \pa^2  v\begin{pmatrix}
1\\
-1\\
1
\end{pmatrix}.
$$
For $\tilde v$ the reflection conditions are 
\begin{align*}
\pa_{3}\tilde v-\frac{1}{2}\pa_{2} \tilde v&=0, \quad \mbox{ if }y_{3}=0,\mbox{ and }(y_1,y_2)\in (0,\infty)^2;\\
\pa_{2}\tilde v-\frac{1}{2}(\pa_{1}\tilde v+\pa_3\tilde v)&=0, \quad \mbox{ if }y_{2}=0\mbox{ and }(y_1,y_3)\in (0,\infty)^2;\\
\pa_{1} \tilde v-\frac{1}{2}\pa_{2} \tilde v&=0, \quad \mbox{ if }y_{1}=0\mbox{ and }(y_2,y_3)\in (0,\infty)^2.
\end{align*}
Thanks to \eqref{eq:deftildev}, this yields \eqref{eq:pdev}-\eqref{eq:refl3} .
\end{proof}

\section{Characterization of the value on the reflection boundary}\label{s.character}
We now characterize the function $v$ via a system of hyperbolic first order PDE.
\subsection{The value of $ v$ for $y_1=y_3$}
We start by characterizing $v$ on the set $y_1=y_3$.
\begin{prop}\label{prop:diag}
The value function $v$ satisfies
 $$v(y_1,y_2,y_1)=V(y_1,y_2),\mbox{ for }y_1\geq 0, \, y_2\geq 0,$$
 where 
 \begin{align}\label{eq:valueondiag}
 V(y_1,y_2):=\frac{1}{2}\E\left[\int_0^\infty e^{-t}\Lambda^{1,(y_1,y_2)}_tdt\right],
 \end{align}
and $\Lambda^{1,(y_1,y_2)}$ is the local time at 0 of the first component of the two dimensional obliquely reflected Brownian Motion 
\  $(Z^{1,(y_1,y_2)}_t,Z^{2, (y_1,y_2)}_t)$ defined by
 \begin{equation}\label{eq:redrpm}
 \begin{split}
dZ^{1,(y_1,y_2)}_t&=dW_t+d\Lambda^{1,(y_1,y_2)}_t-\frac{1}{2}d\Lambda^{2, (y_1,y_2)}_t,\\
dZ^{2,(y_1,y_2)}_t&=-dW_t+d\Lambda^{2, (y_1,y_2)}_t-d\Lambda^{1, (y_1,y_2)}_t.
\end{split}
\end{equation}

Additionally, for all $y\in \R_+^2$, we have 
\begin{align}\label{eq:valueV}
V(y_1,y_2)&=\frac{\sqrt{2}}{2}\cosh(\sqrt{2}y_1)\cosh(\sqrt{2}(y_1+y_2))\arctan\left(e^{-\sqrt{2}(y_1+y_2)}\right)
\\&-\frac{\sqrt{2}}{4}\sinh(\sqrt{2}y_1) \nonumber.
\end{align}
\end{prop}
\begin{proof}
If $y_1=y_3$ it is clear due to the uniqueness of the solution of the oblique reflection problem \eqref{eq:decomY} that for all 
$$Y^{1,y}_t=Y^{3,y}_t\mbox{ for all }t\geq 0 $$
and the couple $(Y^{1,y}_t,Y^{2,y}_t)$ solves the reflection problem \eqref{eq:redrpm}.
Thus, \eqref{eq:valueondiag} holds. 
Additionally, using \eqref{eq:valueondiag} we can derive the following dynamic programming equations for all $(y_1,y_2)\in(0,\infty)^2$,
\begin{align}\label{eq:pdeV}
V(y_1,y_2)-\begin{pmatrix}
1 \\
-1\end{pmatrix}^T \pa^2 V(y_1,y_2)\begin{pmatrix}
1\\
-1
\end{pmatrix}&=0, \\
\pa_1 V(0,y_2)-\pa_2 V(0,y_2)&=-\frac{1}{2}, \label{eq:pdeV2} \\
 \frac{\pa_1 V(y_1,0)}{2}-\pa_2 V(y_1,0)&=0. \label{eq:pdeV3}
 \end{align}
 First, we compute the functions
$$V_1(x):=V(x,0)\mbox{ and }V_2(x):=V(0,x).$$
Let $y_1>0$ and $y_2>0$ and define 
$$\tau:=\tau_1\wedge \tau_2,$$
where $\tau_1:=\inf\{t\geq 0: W_t\leq -y_1\}$ and $\tau_2:=\inf\{t\geq 0: W_t\geq y_2\}$.
Then, by the dynamic programming principle
\begin{align}
V(y_1,y_2)&:=\E\left[e^{-\tau}V(\cY_\tau^1,\cY^2_\tau)\right]\\
&=\E\left[e^{-\tau}\1_{\tau_1<\tau_2}V_2(y_1+y_2)\right]+\E\left[e^{-\tau}\1_{\tau_1>\tau_2}V_1(y_1+y_2)\right]\notag\\
&=\frac{\sinh(\sqrt{2}y_2)}{\sinh(\sqrt{2}(y_1+y_2))}V_2(y_1+y_2)+\frac{\sinh(\sqrt{2}y_1)}{\sinh(\sqrt{2}(y_1+y_2))}V_1(y_1+y_2)\label{eq:expV}
\end{align}
Assuming $V$ is smooth we differentiate this equality in $y_1$, then in the expression we send $y_1\to 0$ for $y_2>0$ fixed to obtain 
$$\pa_1 V(0,y_2)=\frac{\sqrt{2}}{\sinh(\sqrt{2}y_2)}V_1(y_2)-\frac{\sqrt{2}}{\tanh(\sqrt{2}y_2)}V_2(y_2)+V_2'(y_2).$$
One of the main point of the paper is the fact that the equality \eqref{eq:pdeV2} allows us to eliminate $\pa_1 V(0,y_2)$ so that we can write a system of differential equations for $V_1$ and $V_2$ as follows
$$\pa_2 V_2(0,y_2)-\frac{1}{2}=V_2'(y_2)-\frac{1}{2}=\frac{\sqrt{2}}{\sinh(\sqrt{2}y_2)}V_1(y_2)-\frac{\sqrt{2}}{\tanh(\sqrt{2}y_2)}V_2(y_2)+V_2'(y_2).$$
Similarly, differentiating \eqref{eq:expV} in $y_2$ and taking the limit as $y_2\to 0$, we obtain that 
$$\pa_2 V(y_1,0)=-\frac{\sqrt{2}}{\tanh(\sqrt{2}y_1)}V_1(y_1)+\frac{\sqrt{2}}{\sinh(\sqrt{2}y_1)}V_2(y_1)+V_1'(y_1).$$
Additionally, the reflection conditions at \eqref{eq:pdeV3} yield
$$\frac{ V_1'(y_1)}{2}=\pa_2 V(y_1,0)=-\frac{\sqrt{2}}{\tanh(\sqrt{2}y_1)}V_1(y_1)+\frac{\sqrt{2}}{\sinh(\sqrt{2}y_1)}V_2(y_1)+V_1'(y_1).$$
Combining both equalities we find that $(V_1,V_2)$ solves the system
 \begin{align}\label{eq:syst1}
-\frac{1}{2}&=\frac{\sqrt{2}}{\sinh(\sqrt{2}x)}V_1(x)-\frac{\sqrt{2}}{\tanh(\sqrt{2}x)}V_2(x),\\
0&=-\frac{\sqrt{2}}{\tanh(\sqrt{2}x)}V_1(x)+\frac{\sqrt{2}}{\sinh(\sqrt{2}x)}V_2(x)+\frac{V_1'(x)}{2}.
 \end{align}
 Combining the two equalities we obtain that $V_1$ is a solution to 
  \begin{align}\label{eq:V1}
0&=\frac{1}{\cosh(\sqrt{2}x)}-2\sqrt{2}\tanh(\sqrt{2}x)V_1(x)+{V_1'(x)}.
 \end{align}
 Given the antiderivative of the hyperbolic tangent, the solution to the homogeneous part of \eqref{eq:V1} is $x\mapsto \cosh^2(\sqrt {2}x)$. Thus, we solve \eqref{eq:V1} under the form  
 $$V_1(x)=H(x)\cosh^2(\sqrt {2}x),$$
 which imposes $H'(x)=\frac{-1}{\cosh^3(\sqrt {2}x)}$.
 Thus, for some constant $C$, $V_1$ is
$$V_1(x)=\left(C-\frac{1}{\sqrt{2}}\arctan\left(\tanh\left(\frac{x}{\sqrt{2}}\right)\right)\right)\cosh^2(\sqrt{2}x)-\frac{\sinh(\sqrt{2}x)}{2\sqrt{2}}.$$
With the choice $C=\frac{\pi}{4\sqrt{2}}$ we obtain that 
\begin{align}\label{eq:solV1}
V_1(x)=\frac{1}{\sqrt{2}}\left(\frac{\pi}{4}-\arctan\left(\tanh\left(\frac{x}{\sqrt{2}}\right)\right)\right)\cosh^2(\sqrt{2}x)-\frac{\sinh(\sqrt{2}x)}{2\sqrt{2}}
\end{align}
is the unique bounded solution to \eqref{eq:V1}. Indeed, given the properties of the Gudermannian function, and arctan we have
\begin{align*}
\frac{\pi}{4}-\arctan\left(\tanh\left(\frac{x}{\sqrt{2}}\right)\right)&=\frac{\pi}{2}-\arctan\left(e^{\sqrt{2}x}\right)\\
&=\arctan\left(e^{-\sqrt{2}x}\right)=e^{-\sqrt{2}x}+o(e^{-2\sqrt{2}x})
\end{align*}
as $x\to \infty$. 
Thus, as $x\to \infty$, 
\begin{align*}
&\frac{1}{\sqrt{2}}\left(\frac{\pi}{4}-\arctan\left(\tanh\left(\frac{x}{\sqrt{2}}\right)\right)\right)\cosh^2(\sqrt{2}x)-\frac{\sinh(\sqrt{2}x)}{2\sqrt{2}}\\
&\qquad=\frac{1}{\sqrt{2}}\left(e^{-\sqrt{2}x}+o(e^{-2\sqrt{2}x})\right)\left(\frac{1}{4}e ^{2\sqrt{2}x}+O(1)\right)-\frac{1}{4\sqrt{2}}\left(e ^{\sqrt{2}x}+O(1)\right)
\\ &=O(1)
\end{align*}
which shows that \eqref{eq:solV1} is the unique bounded solution to \eqref{eq:V1}.
Injecting this into \eqref{eq:syst1} and further simplifying we obtain that 
\begin{align*}
V_1(x)&=\frac{1}{\sqrt{2}}\arctan\left(e^{-\sqrt{2}x}\right)\cosh^2(\sqrt{2}x)-\frac{\sinh(\sqrt{2}x)}{2\sqrt{2}},\\
V_2(x)&=\frac{1}{\sqrt{2}}\arctan\left(e^{-\sqrt{2}x}\right)\cosh(\sqrt{2}x).
\end{align*}
Thanks to \eqref{eq:expV}, this finally yields \eqref{eq:valueV}.
\end{proof}
\subsection{Deriving a Hyperbolic system to characterize the value on the boundary}

We now return to the computation of $v$ defined at \eqref{eq:defv} on $\R_+^3$. In order to compute $v$ on the whole domain we first characterize its value on the boundary of this domain. For this purpose, we define for $x,y\geq 0$, 
\begin{align}
f(x,y)&= v\left(0,\frac{x}{\sqrt{2}},\frac{y}{\sqrt{2}}\right)\label{eq:deffrh1},\\
r_1(x,y)&= v\left(\frac{x}{\sqrt{2}},0,\frac{y+x}{\sqrt{2}}\right)\label{eq:deffrh2},\\
h(x,y)&= v\left(\frac{y}{\sqrt{2}},\frac{x}{\sqrt{2}},0\right)-\frac{1}{2\sqrt{2}}\left(1+\frac{e^{-2x}}{3}\right),\label{eq:deffrh3}\\
r_2(x,y)&= v\left(\frac{x+y}{\sqrt{2}},0,\frac{x}{\sqrt{2}}\right)-\frac{2}{3\sqrt{2}}e^{-x}\label{eq:deffrh4}.
\end{align}
The next proposition provides a characterization of these functions and allows us to compute the value function everywhere.
\begin{prop}\label{prop:boundary}
The couples $(f,r_1)$ and $(h,r_2)$ solve the same system of hyperbolic equations on $(0,\infty)^2$
\begin{align}
(\pa_x-2\pa_y)f(x,y)&=\frac{2}{\tanh{x}}f(x,y)-\frac{2}{\sinh{x}}r_1(x,y)\label{syst:hf},\\
\pa_x r_1(x,y)&=-\frac{2}{\sinh{x}}f(x,y)+\frac{2}{\tanh{x}}r_1(x,y)\label{syst:hr},
\end{align}
with the compatibility conditions
\begin{align*}
f(0,y)=r_1(0,y),\, h(0,y)=r_2(0,y)\mbox{ for }y> 0,
\end{align*}
and initial conditions
\begin{align*}
f(x,0)&=\frac{1}{\sqrt{2}}\arctan\left(e^{-x}\right)\cosh(x),\\
r_1(x,0)&=\frac{1}{\sqrt{2}}\arctan\left(e^{-x}\right)\cosh^2(x)-\frac{\sinh(x)}{2\sqrt{2}},\\
h(x,0)&=\frac{1}{\sqrt{2}}\arctan\left(e^{-x}\right)\cosh(x)-\frac{1}{2\sqrt{2}}\left(1+\frac{e^{-2x}}{3}\right),\\
r_2(x,0)&=\frac{1}{\sqrt{2}}\arctan\left(e^{-x}\right)\cosh^2(x)-\frac{\sinh(x)}{2\sqrt{2}}-\frac{2}{3\sqrt{2}}e^{-x}\mbox{ for }x> 0.
\end{align*}
\end{prop}
\begin{remark}
In the definition of $h$ and $r_2$ the terms $\frac{1}{2\sqrt{2}}\left(1+\frac{e^{-2x}}{3}\right)$ and $\frac{\sinh(x)}{2\sqrt{2}}-\frac{2}{3\sqrt{2}}e^{-x}$ are subtracted to eliminate $1$ in equation \eqref{eq:with1}. This allow us to study one system of equation with two different initial condition rather than two systems with the same initial condition. 
\end{remark}
\begin{proof}
Proceeding similarly as in \eqref{eq:expV}, we obtain that 
for $0\leq y_1\leq y_3$ we have  
\begin{align}\label{eq:v13}
{ v}(y_1,y_2,y_3)&={ v}(0,y_2+y_1,y_3-y_1)\frac{\sinh(\sqrt{2}y_2)}{\sinh(\sqrt{2}(y_1+y_2))},\notag\\
&\quad\quad+{ v}(y_1+y_2,0,y_3+y_2)\frac{\sinh(\sqrt{2}y_1)}{\sinh(\sqrt{2}(y_1+y_2))},
\end{align}
and for $0\leq y_3\leq y_1$, 
\begin{align}\label{eq:v31}
{ v}(y_1,y_2,y_3)&={ v}(y_1-y_3,y_2+y_3,0)\frac{\sinh(\sqrt{2}y_2)}{\sinh(\sqrt{2}(y_3+y_2))},\notag\\
&\quad\quad+{ v}(y_1+y_2,0,y_3+y_2)\frac{\sinh(\sqrt{2}y_3)}{\sinh(\sqrt{2}(y_3+y_2))}.
\end{align}
Let us first consider the case  $0\leq y_1\leq y_3$. Similarly to the proof of \eqref{prop:diag}, we differentiate \eqref{eq:v13} in $y_1$, and send $y_1$ to $0$, and obtain that 
\begin{equation*}
\begin{split}
\pa_1 v(0,y_2,y_3)&=\pa_2v(0,y_2,y_3)-\pa_3v(0,y_2,y_3)+v(0,y_2,y_3)\frac{-\sqrt{2}}{\tanh(\sqrt{2}y_2)}
\\&+v(y_2,0,y_2+y_3)\frac{\sqrt{2}}{\sinh(\sqrt{2}y_2)}
\end{split}
\end{equation*}
Additionally, the reflection conditions \eqref{eq:refl3} gives 
$$(2\pa_3-\pa_2) v(0,y_2,y_3)=\frac{2\sqrt{2}}{\sinh{(\sqrt{2}y_2)}} v(y_2,0,y_3+y_2)-\frac{2\sqrt{2}}{\tanh{(\sqrt{2}y_2)}} v(0,y_2,y_3)$$
Then we differentiate \eqref{eq:v13} in $y_2$ and send $y_2$ to $0$ to obtain
\begin{align*}
\pa_2 v(y_1,0,y_3)=&\pa_1 v(y_1,0,y_3)+\pa_3 v(y_1,0,y_3)+v(0,y_1,y_3-y_1)\frac{\sqrt{2}}{\sinh(\sqrt{2}y_1)}\\
&+v(y_1,0,y_3)\frac{-\sqrt{2}}{\tanh(\sqrt{2}y_1)}.
\end{align*}
The reflection conditions \eqref{eq:refl2} yields
$$(\pa_1+\pa_3) v(y_1,0,y_3)=\frac{2\sqrt{2}}{\tanh{(\sqrt{2}y_1)}} v(y_1,0,y_3)-\frac{2\sqrt{2}}{\sinh{(\sqrt{2}y_1)}} v(0,y_1,y_3-y_1).$$
Combining both equalities, and write them in $f(x,y)$ and $r_1(x,y)$, we get the desired system:
\begin{align*}
(\pa_x-2\pa_y)f(x,y)&=\frac{2}{\tanh{x}}f(x,y)-\frac{2}{\sinh{x}}r_1(x,y),\\
\pa_x r_1(x,y)&=-\frac{2}{\sinh{x}}f(x,y)+\frac{2}{\tanh{x}}r_1(x,y).
\end{align*}
Let us now consider the case  $0\leq y_3\leq y_1$. 
Following a similar procedure as before, we differentiate \eqref{eq:v31} in $y_2$, and send $y_2$ to $0$ to obtain  
\begin{align*}
\pa_2v(y_1,0,y_3)=&\pa_1v(y_1,0,y_3)+\pa_3v(y_1,0,y_3)+v(y_1-y_3,y_3,0)\frac{\sqrt{2}}{\sinh(\sqrt{2}y_3)}\\
&+v(y_1,0,y_3)\frac{-\sqrt{2}}{\tanh(\sqrt{2}y_3)}.
\end{align*}
Additionally, the reflection conditions \eqref{eq:refl2} gives 
$$(\pa_1+\pa_3) v(y_1,0,y_3)=\frac{2\sqrt{2}}{\tanh{(\sqrt{2}y_3)}} v(y_1,0,y_3)-\frac{2\sqrt{2}}{\sinh{(\sqrt{2}y_3)}} v(y_1-y_3,y_3,0).$$
Then we differentiate \eqref{eq:v31} in $y_3$ and send $y_3$ to $0$ and obtain
\begin{align*}
\pa_3v(y_1,y_2,0)=&\pa_2v(y_1,y_2,0)-\pa_1v(y_1,y_2,0)+v(y_1,y_2,0)\frac{-\sqrt{2}}{\tanh(\sqrt{2}y_2))}\\
&+v(y_1+y_2,0,y_2)\frac{\sqrt{2}}{\sinh(\sqrt{2}y_2)}.
\end{align*}
The reflection conditions \eqref{eq:refl1} gives
\begin{align}\label{eq:with1}
(2\pa_1-\pa_2)v(y_1,y_2,0)=1-\frac{2\sqrt{2}v(y_1,y_2,0)}{\tanh(\sqrt{2}y_2)}+\frac{2\sqrt{2}v(y_1+y_2,0,y_2)}{\sinh(\sqrt{2}y_2)}.
\end{align}
Combining both equalities, and write them in $h(x,y)$ and $r_2(x,y)$, we have the desired system:
\begin{align*}
(\pa_x-2\pa_y)h(x,y)&=\frac{2}{\tanh{x}}h(x,y)-\frac{2}{\sinh{x}}r_2(x,y),\\
\pa_x r_2(x,y)&=-\frac{2}{\sinh{x}}h(x,y)+\frac{2}{\tanh{x}}r_2(x,y).
\end{align*}
The compatibility conditions and initial conditions follows form the change of variable described at the beginning of this section and Proposition~\ref{prop:diag}.
\end{proof}

\subsection{Solving the Hyperbolic system}
Although first order and linear, the system \eqref{syst:hf} can not be directly solved via the method of characteristics since the characteristics for the two equations are not in the same direction. Additionally, we cannot employ methods described in \cite{tsarev2007factorization} and \cite{fusco1996method}.
\subsubsection{Heuristic to find an ansatz of the solution}
We first note that if $f$ is given then thanks to \eqref{syst:hr}, $r$ solves a linear ODE whose unique solutions that is bounded at infinity is  
\begin{align}\label{sol:r}
r_1(x,y)=2\sinh^2(x)\int_x^\infty\frac{f(r,y)}{\sinh^3(r)}dr.
\end{align}
$\{f(x,0)\}_{x\geq 0}$ being given, we can easily obtain $\{r_1(x,0)\}_{x\geq 0}$ by integration. This allows us to compute 
$\{\pa_y f(x,0)\}_{x\geq 0}$ by isolating it in \eqref{syst:hf}.

Since the system does not depend on $y$ we can differentiate in $y$. Thus, we can compute $\{\pa^2_{y} f(x,0)\}_{x\geq 0}$ with a similar procedure if we start with initial condition $\{\pa_y f(x,0)\}_{x\geq 0}$. Then, we can repeat the procedure to compute several derivatives $\{\pa^n_y f(x,0)\}_{x\geq 0}$. 

Additionally thanks to the form of solutions in \cite{iskenderovinverse}, we expect that the solutions $f$ and $r$ are functions of $x+\frac{y}{2}$ and $\frac{y}{2}$. 
Combining this with the computation of the derivatives $\{\pa^n_y f(x,0)\}_{x\geq 0}$ we conjecture that 
\begin{align*}
f(x,y)=&h_1\left(\frac{y}{2}\right)\arctan\left(e^{-x-\frac{y}{2}}\right)\cosh\left(x+\frac{y}{2}\right)\\
&+h_2\left(\frac{y}{2}\right)\arctanh\left(e^{-x-\frac{y}{2}}\right)\sinh\left(x+\frac{y}{2}\right)+h_3\left(\frac{y}{2}\right)
\end{align*}
with the condition 
$$h_1(0)=\frac{1}{\sqrt{2}}, \, h_2(0)=h_3(0)=0.$$

\subsubsection{Solution to the systems}
Given the ansatz for $f$, one can integrate \eqref{sol:r} to find that $r$ then \eqref{syst:hf} leads to \footnote{This computation could be extremely tedious by hand. We have checked the identity with Mathematica V11. The code for this verification and other tedious computations are provided in \cite{BEYMath}.}  
\begin{align}
&\notag 2 \arctanh(e^{-x})\sinh(x)\left(\coth\left(\frac{y}{2}\right)h_2\left(\frac{y}{2}\right)+2h_3\left(\frac{y}{2}\right)+h_1\left(\frac{y}{2}\right)\tanh\left(\frac{y}{2}\right)\right)\\
\notag&+\arctan\left(e^{-x-\frac{y}{2}}\right)\cosh\left(x+\frac{y}{2}\right)\left(2h_1\left(\frac{y}{2}\right)\tanh\left(\frac{y}{2}\right)-h'_1\left(\frac{y}{2}\right)\right)\\
\notag&+\arctanh\left(e^{-x-\frac{y}{2}}\right)\sinh\left(x+\frac{y}{2}\right)\left(2h_2\left(\frac{y}{2}\right)\coth\left(\frac{y}{2}\right)-h'_2\left(\frac{y}{2}\right)\right)\\
&-h_1\left(\frac{y}{2}\right)-h_2\left(\frac{y}{2}\right)-h_3'\left(\frac{y}{2}\right)=0.\label{eq:solansatz}
\end{align}
Setting the second and the third lines to $0$, we solve the ODE obtained for $h_1$ and $h_2$ with the initial condition to obtain that 
$$h_1(y)=\frac{1}{\sqrt{2}}\cosh^2(y)\mbox{ and }h_2(y)=C\sinh^2(y)\mbox{ for some constant }C.$$ 
Injecting this to the first line, the term in parentheses in the first line becomes 
$$\left(\frac{C}{2}+\frac{1}{2\sqrt{2}}\right)\sinh\left(y\right)+2h_3\left(\frac{y}{2}\right).$$
This allows us to identify 
$$h_3\left({y}\right)=-\frac{C\sqrt{2}+1}{4\sqrt{2}}\sinh\left(2y\right).$$
Thus, to satisfy \eqref{eq:solansatz} we need
$$\frac{1}{\sqrt{2}}\cosh^2(y)+C\sinh^2(y)-\frac{C\sqrt{2}+1}{2\sqrt{2}}\cosh\left(2y\right)=0$$
which is satisfied for $C=\frac{1}{\sqrt{2}}.$
Thus, we obtain $f$ as
\begin{align}\label{eq:f}
f(x,y)&=\frac{1}{\sqrt{2}}\left(\arctan(e^{-x-\frac{y}{2}})\cosh(x+\frac{y}{2})\cosh^2(\frac{y}{2})\right)\\
&+\frac{1}{\sqrt{2}}\left(\arctanh{(e^{-x-\frac{y}{2}})}\sinh(x+\frac{y}{2})\sinh^2(\frac{y}{2})-\frac{1}{2}\sinh(y)\right). \notag
\end{align}
Injecting this expression in \eqref{sol:r} we obtain
\begin{align}\label{eq:r1}
r_1(x,y)&=\frac{1}{\sqrt{2}}\left(\arctan(e^{-x-\frac{y}{2}})\cosh^2(x+\frac{y}{2})\cosh(\frac{y}{2})\right)\\
&+\frac{1}{\sqrt{2}}\left(\arctanh{(e^{-x-\frac{y}{2}})}\sinh^2(x+\frac{y}{2})\sinh(\frac{y}{2})-\frac{1}{2}\sinh(x+y)\right).\notag
\end{align}
Using the same method we can also solve the system \eqref{syst:hf}-\eqref{syst:hr} with initial condition 
$$\left(\frac{1}{2\sqrt{2}}\left(1+\frac{e^{-2x}}{3}\right), \frac{2}{3\sqrt{2}}e^{-x}\right),$$ then using the linearity of the system subtract this from $(f,r)$ to obtain 
\begin{align}\label{eq:h}
h(x,y)&=\frac{1}{\sqrt{2}}\left(\arctan(e^{-x-\frac{y}{2}})\cosh(x+\frac{y}{2})\cosh^2(\frac{y}{2})-\frac{1}{2}-\frac{e^{-2x}}{6}\right)\\
&-\frac{1}{\sqrt{2}}\arctanh{(e^{-x-\frac{y}{2}})}\sinh(x+\frac{y}{2})\sinh^2(\frac{y}{2}), \notag\\
\label{eq:r2}r_2(x,y)&=\frac{1}{\sqrt{2}}\left(\arctan(e^{-x-\frac{y}{2}})\cosh^2(x+\frac{y}{2})\cosh(\frac{y}{2})-2\frac{\cosh (x)}{3} \right)\\
&\notag-\frac{1}{\sqrt{2}}\left(\arctanh{(e^{-x-\frac{y}{2}})}\sinh^2(x+\frac{y}{2})\sinh(\frac{y}{2})-\frac{\sinh (x)}{6} \right).
\end{align}
The reader may find in \cite{BEYMath}, the Mathematica code to check that \eqref{eq:f}-\eqref{eq:r2} provides solutions to the system \eqref{syst:hf} and \eqref{syst:hr}.
Combining \eqref{eq:f}, \eqref{eq:r1}, \eqref{eq:h} and \eqref{eq:r2}, we now give the expression of $v$. 
\begin{prop}
The function $v$ defined at \eqref{eq:defv} is given by 
\begin{align*}
 v(y_1,y_2,y_3)=&-\frac{\sqrt2}{4}\sinh(\sqrt{2}y_3)\\
 &+\frac{\sqrt{2}}{2}\arctan\left(e^{-\frac{y_1+2y_2+y_3}{\sqrt{2}}}\right)\\
 &\cosh\left(\frac{-y_1+y_3}{\sqrt{2}}\right)\cosh\left(\frac{y_1+2y_2+y_3}{\sqrt{2}}\right)\cosh\left(\frac{y_1+y_3}{\sqrt{2}}\right)\notag\\
   &+\frac{\sqrt{2}}{2}\arctanh\left(e^{-\frac{y_1+2y_2+y_3}{\sqrt{2}}}\right)\\
   &\sinh\left(\frac{-y_1+y_3}{\sqrt{2}}\right)\sinh\left(\frac{y_1+2y_2+y_3}{\sqrt{2}}\right)\sinh\left(\frac{y_1+y_3}{\sqrt{2}}\right)\notag
\end{align*}
\end{prop}
\begin{remark}
For reader's convenience we provide in \cite{BEYMath} the Mathematica code to check that this expression provides a solution to the equations \eqref{eq:pdev} and \eqref{eq:refl1}-\eqref{eq:refl3}.
\end{remark}
\begin{proof}
The proof is a direct consequence of identities \eqref{eq:v13}-\eqref{eq:v31} and \eqref{eq:f}-\eqref{eq:r2}.We inject $f(x,y)$ and $r_1(x,y)$ to obtain $v$ for $0\leq y_1\leq y_3$, i.e. 
{\footnotesize
\begin{align*}
 &v(y_1,y_2,y_3)\\
 &=f(\sqrt{2}(y_1+y_2)),\sqrt{2}(y_3-y_1)\frac{\sinh(\sqrt{2}y_2)}{\sinh(\sqrt{2}(y_1+y_2))}
 \\&+r_1(\sqrt{2}(y_1+y_2),\sqrt{2}(y_3-y_1))\frac{\sinh(\sqrt{2}y_1)}{\sinh(\sqrt{2}(y_1+y_2))}\\
 &=\frac{\sqrt{2}}{2}\arctan(e^{-\frac{y_1+2y_2+y_3}{\sqrt{2}}})\cosh(\frac{-y_1+y_3}{\sqrt{2}})\cosh^2(\frac{y_1+2y_2+y_3}{\sqrt{2}})\\&\times \csch(\sqrt{2}(y_1+y_2))\sinh(\sqrt{2}y_1)\\
     &+\frac{\sqrt{2}}{2}\arctan(e^{-\frac{y_1+2y_2+y_3}{\sqrt{2}}})\cosh^2(\frac{-y_1+y_3}{\sqrt{2}})\cosh(\frac{y_1+2y_2+y_3}{\sqrt{2}})\\& \times \csch(\sqrt{2}(y_1+y_2))\sinh(\sqrt{2}y_2)\\
    &-\frac{\sqrt{2}}{4}\csch (\sqrt{2}(y_1+y_2))\sinh(\sqrt{2}y_2)\sinh(\sqrt{2}(-y_1+y_3))\\
    &-\frac{\sqrt{2}}{4}\csch (\sqrt{2}(y_1+y_2))\sinh(\sqrt{2}y_1)\sinh(\sqrt{2}(y_2+y_3))\\
    &+\frac{\sqrt{2}}{2}\arctanh(e^{-\frac{y_1+2y_2+y_3}{\sqrt{2}}})\csch(\sqrt{2}(y_1+y_2))\sinh(\sqrt{2}y_2)\\& \times\sinh^2(\frac{-y_1+y_3}{\sqrt{2}})\sinh(\frac{y_1+2y_2+y_3}{\sqrt{2}})\\
    &+\frac{\sqrt{2}}{2}\arctanh(e^{-\frac{y_1+2y_2+y_3}{\sqrt{2}}})\csch(\sqrt{2}(y_1+y_2))\sinh(\sqrt{2}y_1)\\& \times \sinh(\frac{-y_1+y_3}{\sqrt{2}})\sinh^2(\frac{y_1+2y_2+y_3}{\sqrt{2}}).
\end{align*}}
And injecting $h(x,y)$ and $r_2(x,y)$ we obtain $v$ for $0\leq y_3\leq y_1$
{\footnotesize
\begin{align*}
&v(y_1,y_2,y_3)\\
&=\left(h(\sqrt{2}(y_2+y_3),\sqrt{2}(y_1-y_3))+\frac{1}{2\sqrt{2}}(1+\frac{e^{2\sqrt{2}(y_2+y_3)}}{3})\right)\frac{\sinh(\sqrt{2}y_2)}{\sinh(\sqrt{2}(y_3+y_2))}\\
&+\left(r_2(\sqrt{2}(y_2+y_3),\sqrt{2}(y_1-y_3)+\frac{2}{3\sqrt{2}}e^{-\sqrt{2}(y_2+y_3)})\right)\frac{\sinh(\sqrt{2}y_3)}{\sinh(\sqrt{2}(y_3+y_2))}\\
&= \frac{\sqrt{2}}{2}\arctan(e^{-\frac{y_1+2y_2+y_3}{\sqrt{2}}})\cosh^2(\frac{y_1-y_3}{\sqrt{2}})\cosh(\frac{y_1+2y_2+y_3}{\sqrt{2}})\csch(\sqrt{2}(y_3+y_2))\sinh(\sqrt{2}y_2)\\
     &+\frac{\sqrt{2}}{2}\arctan(e^{-\frac{y_1+2y_2+y_3}{\sqrt{2}}})\cosh(\frac{y_1-y_3}{\sqrt{2}})\cosh^2(\frac{y_1+2y_2+y_3}{\sqrt{2}})\csch(\sqrt{2}(y_3+y_2))\sinh(\sqrt{2}y_3)\\
    &-\frac{\sqrt{2}}{4} \sinh(\sqrt{2}y_3)\\
    &-\frac{\sqrt{2}}{2}\arctanh(e^{-\frac{y_1+2y_2+y_3}{\sqrt{2}}})\csch(\sqrt{2}(y_3+y_2))\sinh(\sqrt{2}y_2)\sinh^2(\frac{y_1-y_3}{\sqrt{2}})\sinh(\frac{y_1+2y_2+y_3}{\sqrt{2}})\\
    &-\frac{\sqrt{2}}{2}\arctanh(e^{-\frac{y_1+2y_2+y_3}{\sqrt{2}}})\csch(\sqrt{2}(y_3+y_2))\sinh(\sqrt{2}y_3)\sinh(\frac{y_1-y_3}{\sqrt{2}})\sinh^2(\frac{y_1+2y_2+y_3}{\sqrt{2}})
\end{align*}}
Note that these expressions can be simplified and combined into one expression on the whole space $0\leq y_1,y_2,y_3$
\begin{align*}
&v(y_1,y_2,y_3)
\\&= \frac{\sqrt{2}}{2}\arctan\left(e^{-\frac{y_1+2y_2+y_3}{\sqrt{2}}}\right)\cosh\left(\frac{-y_1+y_3}{\sqrt{2}}\right)\cosh\left(\frac{y_1+2y_2+y_3}{\sqrt{2}}\right)\cosh\left(\frac{y_1+y_3}{\sqrt{2}}\right)\notag\\
    &-\frac{\sqrt2}{4}\sinh(\sqrt{2}y_3)\notag\\
    &+\frac{\sqrt{2}}{2}\arctanh\left(e^{-\frac{y_1+2y_2+y_3}{\sqrt{2}}}\right)\sinh\left(\frac{-y_1+y_3}{\sqrt{2}}\right)\sinh\left(\frac{y_1+2y_2+y_3}{\sqrt{2}}\right)\sinh\left(\frac{y_1+y_3}{\sqrt{2}}\right).\notag
\end{align*}
\end{proof}

We will close this section by giving a minimum principle for the supersolutions of the system \eqref{syst:hf}-\eqref{syst:hr}, which we will need in the next section when proving our main result. 

\begin{prop}\label{prop:comp}
Let $F,R:[0,\infty)^2\mapsto \R$ be functions that are continuous on their domain and continuously differentiable in the interior of their domain. Assume that for all $x,y\geq 0$, 
 $$F(x,0)\geq 0,\, F(0,y)\geq 0,\,{\liminf_{r^2+s^2\to\infty}F(r,s)\geq 0},\mbox{ and }\lim_{r\to\infty}R(r,y)=0.$$
Assume also that $F,R$ are supersolution of \eqref{syst:hf}-\eqref{syst:hr} in the sense
\begin{align}\label{eq:Fsuper}
(\pa_x-2\pa_y)F(x,y)&\leq \frac{2}{\tanh{x}}F(x,y)-\frac{2}{\sinh{x}}R(x,y)\\
\pa_x R(x,y)&\leq-\frac{2}{\sinh{x}}F(x,y)+\frac{2}{\tanh{x}}R(x,y)\label{eq:Rsuper}
\end{align}
Then $F(x,y)\geq 0$ and $R(x,y)\geq 0$ for all $x,y\geq 0$.
\end{prop}
\begin{proof}To obtain a contradiction we first assume that $F$ is negative at some point on its domain. Therefore, by the values of this function on the boundary of the domain, its minimum on $[0,\infty)^2$ is achieved and there exists $(x_0,y_0)\in(0,\infty)^2$
and $\d>0$ such that 
$$\inf_{x,y\in [0,\infty)}F(x,y)=F(x_0,y_0)=-\d<0.$$
Thanks to \eqref{eq:Rsuper} we can write 
$$\pa_x R(x,y)=-\frac{2}{\sinh{x}}F(x,y)+\frac{2}{\tanh{x}}R(x,y)-P(x,y)$$
for some $P\geq 0$ and continuous. 
We solve this ODE to obtain similarly to \eqref{sol:r} that 
\begin{align}
R(x,y)&=\sinh^2(x)\int_x^\infty \frac{2F(r,y)}{\sinh^3(r)} +\frac{P(r,y)}{\sinh^2(r)}dr\geq 2\sinh^2(x)\int_x^\infty \frac{F(r,y)}{\sinh^3(r)} dr\label{ineq:RF}\\
&\geq 2\inf_{r\in[x,\infty]}F(r,y)\sinh^2(x)\int_x^\infty \frac{1}{\sinh^3(r)} dr.\notag
\end{align}
We have the identity 
$$2\sinh^2(x)\int_x^\infty \frac{1}{\sinh^3(r)} dr=\cosh(x)-2\arctanh(e^{-x})\sinh^2(x)\in[0,1]\mbox{ for all }x>0.$$
Thus, 
\begin{align}\label{eq:inekR}
R(x_0,y_0)\geq 2\inf_{r\in[x_0,\infty]}F(r,y_0)\sinh^2(x_0)\int_{x_0}^\infty \frac{1}{\sinh^3(r)} dr\geq \inf_{r\in[x_0,\infty]}F(r,y_0)=-\d
\end{align}
where the last inequality is due to the fact that 
$$\inf_{r\in[x_0,\infty]}F(r,y_0)=-\d<0.$$
The minimality of $F$ at $(x_0,y_0)\in (0,\infty)^2$ and the differentiability of $F$ (which implies that $\partial_x F (x_0,y_0)=\partial_y F(x_0,y_0)=0$) combined with \eqref{eq:Fsuper} allows us to claim that 
$${\cosh(x_0)}F(x_0,y_0)\geq R(x_0,y_0).$$
Then, the inequality \eqref{eq:inekR} yields
$$-\d{\cosh(x_0)}={\cosh(x_0)}F(x_0,y_0)\geq R(x_0,y_0)\geq -\d$$
which is in contradiction with $x_0>0$. 
Thus, 
$F\geq 0$. 
Combining this inequality with \eqref{ineq:RF}, we obtain that $R\geq 0$. 
\end{proof}
\section{Regularity of $u$ and proof of the main theorems}\label{s.proof}
In this section we use the expression of $v$ to define the candidate solution to the PDE \eqref{eq:pdegeo}. Let $\cW_4:=\{x\in \R^d:x_1<x_2<x_3<x_4\}$ and define 
\begin{equation}
\begin{split}\label{def:U}
U:x\in\cW_4\mapsto& x_4+v(x_2-x_1,x_3-x_2,x_4-x_3)=x_4-\frac{\sqrt{2}}{4} \sinh(\sqrt{2}(x_4-x_3))\\
&+\frac{\sqrt{2}}{2} \arctan\left(e^{\frac{x_1+x_2-x_3-x_4}{\sqrt{2}}}\right)\cosh\left(\frac{x_1-x_2+x_3-x_4}{\sqrt{2}}\right)\\
&\quad\quad\quad\cosh\left(\frac{-x_1+x_2+x_3-x_4}{\sqrt{2}}\right)\cosh\left(\frac{-x_1-x_2+x_3+x_4}{\sqrt{2}}\right)\\
&+\frac{\sqrt{2}}{2}\arctanh\left(e^{\frac{x_1+x_2-x_3-x_4}{\sqrt{2}}}\right)\sinh\left(\frac{x_1-x_2+x_3-x_4}{\sqrt{2}}\right)\\
&\quad\quad\quad\sinh\left(\frac{-x_1+x_2+x_3-x_4}{\sqrt{2}}\right)\sinh\left(\frac{-x_1-x_2+x_3+x_4}{\sqrt{2}}\right)
\end{split}
\end{equation}
so that 
\begin{align}\label{eq:defuU}
u(x)=U(x^{(1)},x^{(2)},x^{(3)},x^{(4)})\mbox{ for }x\in \R^4.
\end{align}
We give the following proposition for the regularity of $u$ and $U$. 

\begin{prop}\label{prop:regularity}
$U$ has a $C^2$ extension to $\bar \cW_4$ and the extension satisfies for all $x\in \bar \cW_4$, 
\begin{align}
\pa_1U(x_1,x_1,x_3,x_4)=\pa_2U(x_1,x_1,x_3,x_4),\label{eq:verifbound1}\\
\pa_2U(x_1,x_2,x_2,x_4)=\pa_3U(x_1,x_2,x_2,x_4),\label{eq:verifbound2}\\
\pa_3U(x_1,x_2,x_3,x_3)=\pa_4U(x_1,x_2,x_3,x_3).\label{eq:verifbound3}
\end{align}
Additionally, $u$ defined by \eqref{eq:u} is $C^2$ on $\R^4$ and $U$ satisfies 
\begin{align}\label{eq:PDEUU}
0&=U(x)-\Phi(x)-\frac{1}{2}\begin{pmatrix}
0\\
1\\
0\\
1
\end{pmatrix}^T \pa^2 U(x)\begin{pmatrix}
0\\
1\\
0\\
1
\end{pmatrix}\mbox{ for all }x\in \cW_4.
\end{align}
 \end{prop}
\begin{remark}
As needed for the smoothness of $u$, $U$ is symmetric in its variables. 
\end{remark}
\begin{proof}
The main problem with the existence of the extension of $U$ is the fact that the function 
$z\mapsto\arctanh(e^z)$ has a singularity at $0$. Thus, the $C^2$ extension a priori only exists whenever all the components are not equal to each other. 

For the points where all the components are equal to each other we use the fact that $\arctanh(e^z)\sinh(z)\to 0$ as $z\downarrow 0$.
Thus, the last two lines of \eqref{def:U} goes to $0$ as $x$ converges to a point whose components are equal. This shows that there is a continuous extension of $U$ to $\bar \cW_4$. 

To show that the extension is $C^1$ it is now sufficient to show that all partial derivatives admits finite limits as we take the limit to the boundary of $\cW_4$, in particular, when $x_1=x_2=x_3=x_4$. First, we observe that 
\begin{align}\label{eq:G}
    G(x_1,x_2,x_3,x_4)=& x_4-\frac{\sqrt{2}}{4} \sinh(\sqrt{2}(x_4-x_3))\\
&+\frac{\sqrt{2}}{2} \arctan\left(e^{\frac{x_1+x_2-x_3-x_4}{\sqrt{2}}}\right)\cosh\left(\frac{x_1-x_2+x_3-x_4}{\sqrt{2}}\right)\notag\\
&\quad\quad\quad\cosh\left(\frac{-x_1+x_2+x_3-x_4}{\sqrt{2}}\right)\cosh\left(\frac{-x_1-x_2+x_3+x_4}{\sqrt{2}}\right)\notag
\end{align}
is analytic everywhere so we only need to consider the behavior of 
\begin{align*}
    T(x_1,x_2,x_3,x_4)=&\frac{\sqrt{2}}{2}\arctanh\left(e^{\frac{x_1+x_2-x_3-x_4}{\sqrt{2}}}\right)\sinh\left(\frac{x_1-x_2+x_3-x_4}{\sqrt{2}}\right)\\
    &\quad\quad\quad\sinh\left(\frac{-x_1+x_2+x_3-x_4}{\sqrt{2}}\right)\sinh\left(\frac{-x_1-x_2+x_3+x_4}{\sqrt{2}}\right)
\end{align*}
at a point satisfying $x_1=x_2=x_3=x_4$.
By chain rule, the fist order partial derivatives of $U$ are linear combinations of the following 4 terms:
\begin{align*}
    t_1(x_1,x_2,x_3,x_4)=&\arctanh\left(e^{\frac{x_1+x_2-x_3-x_4}{\sqrt{2}}}\right)\sinh\left(\frac{x_1-x_2+x_3-x_4}{\sqrt{2}}\right)\\
    &\quad\quad\quad\sinh\left(\frac{-x_1+x_2+x_3-x_4}{\sqrt{2}}\right)\cosh\left(\frac{-x_1-x_2+x_3+x_4}{\sqrt{2}}\right)\\
    t_2(x_1,x_2,x_3,x_4)=&\arctanh\left(e^{\frac{x_1+x_2-x_3-x_4}{\sqrt{2}}}\right)\sinh\left(\frac{x_1-x_2+x_3-x_4}{\sqrt{2}}\right)\\
    &\quad\quad\quad\cosh\left(\frac{-x_1+x_2+x_3-x_4}{\sqrt{2}}\right)\sinh\left(\frac{-x_1-x_2+x_3+x_4}{\sqrt{2}}\right)\\
    t_3(x_1,x_2,x_3,x_4)=&\arctanh\left(e^{\frac{x_1+x_2-x_3-x_4}{\sqrt{2}}}\right)\cosh\left(\frac{x_1-x_2+x_3-x_4}{\sqrt{2}}\right)\\
    &\quad\quad\quad\sinh\left(\frac{-x_1+x_2+x_3-x_4}{\sqrt{2}}\right)\sinh\left(\frac{-x_1-x_2+x_3+x_4}{\sqrt{2}}\right)\\
    t_4(x_1,x_2,x_3,x_4)=&\frac{e^{\frac{x_1+x_2-x_3-x_4}{\sqrt{2}}}}{1-e^{\sqrt{2}(x_1+x_2-x_3-x_4)}}\sinh\left(\frac{x_1-x_2+x_3-x_4}{\sqrt{2}}\right)\\
    &\quad\quad\quad\sinh\left(\frac{-x_1+x_2+x_3-x_4}{\sqrt{2}}\right)\sinh\left(\frac{-x_1-x_2+x_3+x_4}{\sqrt{2}}\right).
\end{align*}
In $\cW$, as $x_1<x_2<x_3<x_4$,  we have the inequalities 
\begin{align*}
0&\geq x_1-x_2+x_3-x_4 \geq x_1+x_2-x_3-x_4 \\
-x_1-x_2+x_3+x_4&\geq -x_1+x_2+x_3-x_4 \geq x_1+x_2-x_3-x_4.
\end{align*}
Combined with the equality $|\sinh(x)|=\sinh(|x|)$, these inequalities yield
\begin{align*}
    \left|\sinh\left(\frac{x_1-x_2+x_3-x_4}{\sqrt{2}}\right)\right|\leq\left|\sinh\left(\frac{x_1+x_2-x_3-x_4}{\sqrt{2}}\right)\right|\\
     \left|\sinh\left(\frac{-x_1+x_2+x_3-x_4}{\sqrt{2}}\right)\right|\leq\left|\sinh\left(\frac{x_1+x_2-x_3-x_4}{\sqrt{2}}\right)\right|.
\end{align*}
Using the observation that $\arctanh(e^z)\sinh(z)\to 0$ as $z\downarrow 0$ one more time, and the limit
 $\frac{\sinh(z/\sqrt{2})}{1-e^{\sqrt{2}z}} \to \frac{1}{2}$, as $z\downarrow 0$ we can conclude that each of  $t_1$, $t_2$, $t_3$, and $t_4$ $\to 0$ as $x$ converge to a point where components are equal to each other. Thus, we have showed that $T$ has a $C^1$ extension to $\bar \cW_4 $ and in fact all its first order partial derivatives are $0$ on $x_1=x_2=x_3=x_4$.
 
Similarly, using these observations, one can also show that all the second order partial derivatives of $U$ have continuous extension on $x_1=x_2=x_3=x_4$ and all second order partial derivatives of $T$ are $0$ on $x_1=x_2=x_3=x_4$ as well. 

We now use the reflection conditions \eqref{eq:refl1}, \eqref{eq:refl2}, and \eqref{eq:refl3} to show that on the boundaries $x_1=x_2$, $x_2=x_3$, $x_3=x_4$, the first order partial derivatives of $U$ satisfy \eqref{eq:verifbound1}, \eqref{eq:verifbound2}, and \eqref{eq:verifbound3}.

Since $U(x_1,x_2,x_3,x_4)=x_4+v(x_2-x_1,x_3-x_2,x_4-x_3)$ using \eqref{eq:refl3} we obtain that
\begin{align*}
&\pa_1U(x_1,x_1,x_3,x_4)-\pa_2U(x_1,x_1,x_3,x_4)\\
&=-\pa_1v(0,x_3-x_1,x_4-x_3)-\left(\pa_1v(0,x_3-x_1,x_4-x_3)-\pa_2v(0,x_3-x_1,x_4-x_3)\right)\\
&=-2\pa_1v(0,x_3-x_1,x_4-x_3)+\pa_2v(0,x_3-x_1,x_4-x_3)=0.
\end{align*}
Using \eqref{eq:refl2} we obtain
\begin{align*}
&\pa_2U(x_1,x_2,x_2,x_4)-\pa_3U(x_1,x_2,x_2,x_4)\\
&=\pa_1v(x_2-x_1,0,x_4-x_2)-\pa_2v(x_2-x_1,0,x_4-x_2)\\
&\quad\quad-(\pa_2v(x_2-x_1,0,x_4-x_2)-\pa_3v(x_2-x_1,0,x_4-x_2))\\
&=\pa_1v(x_2-x_1,0,x_4-x_2)-2\pa_2v(x_2-x_1,0,x_4-x_2)+\pa_3v(x_2-x_1,0,x_4-x_2)=0.
\end{align*}
On the other hand \eqref{eq:refl1} gives
\begin{align*}
&\pa_3U(x_1,x_2,x_3,x_3)-\pa_4U(x_1,x_2,x_3,x_3)\\
&=\pa_2v(x_2-x_1,x_3-x_2,0)-\pa_3v(x_2-x_1,x_3-x_2,0)-(1+\pa_3v(x_2-x_1,x_3-x_2,0))\\
&=-1+\pa_2v(x_2-x_1,x_3-x_2,0)-2\pa_3v(x_2-x_1,x_3-x_2,0)=0.
\end{align*}

Thus, $U$ has a $C^2$ extension to $\bar \cW_4$, its first order partial derivatives satisfy \eqref{eq:verifbound1}-\eqref{eq:verifbound3} and the first two order of partial derivatives of $T$ are $0$ on $x_1=x_2=x_3=x_4$.

We now show that $u$ defined by \eqref{eq:defuU} or \eqref{eq:u} is $C^2$ on $\R^4$. The smoothness of $U$ and the equalities \eqref{eq:verifbound1}-\eqref{eq:verifbound3} implies that $u$ is $C^1$. 
In order to show that $u$ is $C^2$ we need to show that for any point $x\in\bar \cW_4$ that has two components $x_i,x_j$ equal, the Hessian of $U$ is symmetric in $x_i$ and $x_j$. This is implied by the conditions
\begin{align}
&\pa_{1,1}U(x_1,x_1,x_3,x_4)=\pa_{2,2}U(x_1,x_1,x_3,x_4),\label{eq:11}\\
&\pa_{1,2}U(x_1,x_1,x_3,x_4)=\pa_{2,1}U(x_1,x_1,x_3,x_4),\label{eq:parder1}\\
&\pa_{1,3}U(x_1,x_1,x_3,x_4)=\pa_{2,3}U(x_1,x_1,x_3,x_4),\label{eq:cross1}\\
&\pa_{1,4}U(x_1,x_1,x_3,x_4)=\pa_{2,4}U(x_1,x_1,x_3,x_4),\label{eq:cross2}\\
&\pa_{2,2}U(x_1,x_2,x_2,x_4)=\pa_{3,3}U(x_1,x_2,x_2,x_4),\label{eq:22}\\
&\pa_{2,3}U(x_1,x_2,x_2,x_4)=\pa_{3,2}U(x_1,x_2,x_2,x_4),\label{eq:parder2}\\
&\pa_{2,1}U(x_1,x_2,x_2,x_4)=\pa_{3,1}U(x_1,x_2,x_2,x_4),\label{eq:cross3}\\
&\pa_{2,4}U(x_1,x_2,x_2,x_4)=\pa_{3,4}U(x_1,x_2,x_2,x_4),\label{eq:cross4}\\
&\pa_{3,3}U(x_1,x_2,x_3,x_3)=\pa_{4,4}U(x_1,x_2,x_3,x_3),\label{eq:33}\\
&\pa_{3,4}U(x_1,x_2,x_3,x_3)=\pa_{4,3}U(x_1,x_2,x_3,x_3),\label{eq:parder3}\\
&\pa_{3,1}U(x_1,x_2,x_3,x_3)=\pa_{4,1}U(x_1,x_2,x_3,x_3),\label{eq:cross5}\\
&\pa_{3,2}U(x_1,x_2,x_3,x_3)=\pa_{4,2}U(x_1,x_2,x_3,x_3)\label{eq:cross6}
\end{align}
for $x\in \bar \cW_4$. Thanks to the smoothness of $U$ on $\bar \cW_4$, in fact, we only need these equalities for $x\in \cW_4$.

Note that for $x\in \cW_4$, around each of the points
$$(x_1,x_1,x_3,x_4),\,(x_1,x_2,x_2,x_4),\mbox{ and }(x_1,x_2,x_3,x_3)$$
there exists a neighborhood such that the expression defining $U$ is analytical on this neighborhood. 
Thus, we can apply Schwarz Theorem to obtain \eqref{eq:parder1}, \eqref{eq:parder2} and \eqref{eq:parder3}.
The remaining conditions \eqref{eq:cross1}, \eqref{eq:cross2}, \eqref{eq:cross3}, \eqref{eq:cross4}, \eqref{eq:cross5}, and \eqref{eq:cross6} on cross derivatives are consequences of differentiation of \eqref{eq:verifbound1}-\eqref{eq:verifbound3}.
To show \eqref{eq:11}, we differentiate \eqref{eq:verifbound1} in $x_1$ then subtract \eqref{eq:parder1} to obtain 
$$\pa_{1,1}U(x_1,x_1,x_3,x_4)=\pa_{2,2}U(x_1,x_1,x_3,x_4).$$
Repeating the same procedure with \eqref{eq:verifbound2}, $x_2$ and \eqref{eq:parder2} then with \eqref{eq:verifbound3}, $x_3$ and \eqref{eq:parder3} we obtain \eqref{eq:22} and \eqref{eq:33} which concludes the proof. 
\end{proof}

\subsection{Proof of Theorem \ref{thm:main}}

The expansion of $u$, in \eqref{eq:expofu}, can be found by taking the second order Taylor expansion of $G$ defined in \eqref{eq:G}\footnote{The code of the computation is available in \cite{BEYMath}.}. Note that as discussed in the proof of Proposition \ref{prop:regularity}, the first two derivatives of $u$ and $G$ are equal at $0$ and hence the lack of smoothness of the $\arctanh$ does not contribute to the second order derivative at the origin.

We now show that $U$ defined in \eqref{def:U} solves \eqref{eq:pdegeo} on $\cW_4$ which implies by continuity of the derivatives  that $u$ solves the same PDE on $\R^4$. 
By direct computation\footnote{The code of the computation is available in \cite{BEYMath}} we have that for all $x\in \cW_4$ we have
\begin{align*}
0&=U(x)-\Phi(x)-\frac{1}{2}\begin{pmatrix}
0\\
1\\
0\\
1
\end{pmatrix}^T \pa^2 U(x)\begin{pmatrix}
0\\
1\\
0\\
1
\end{pmatrix},\\
0&=U(x)-\Phi(x)-\frac{1}{2}\begin{pmatrix}
0\\
1\\
1\\
0
\end{pmatrix}^T \pa^2 U(x)\begin{pmatrix}
0\\
1\\
1\\
0
\end{pmatrix}.
\end{align*} 
The function $U$ also satisfies the equality \eqref{translate}. Using its smoothness, we obtain 
\begin{align}\label{eq:deriv}
1=\frac{U(x+\lambda (e_1+e_2+e_3+e_4))-U(x)}{\lambda}\to \sum_{i=1}^4 \pa_i U(x)\mbox{ as }\lambda \to 0.
\end{align}
Note that   
$1= \sum_{i=1}^4 \pa_i U(x)$
implies  
$$\pa^2 U(x)(e_1+e_2+e_3+e_4)=0\mbox{ for all }x\in \R^4.$$
Therefore, for all $J\in P(N)$, we have that 
$$e_{J}^\top \pa^2 U(x)e_{J}-e_{J^c}^\top \pa^2 U(x)e_{J^c}=(e_{J}^\top-e_{J^c}^\top) \pa^2 U(x)(e_{J}+e_{J^c})=0.$$
Thus, if $J$ is a maximizer of the Hamiltonian $\sup_{J\in P(N)} e_J^\top \pa^2 u(x)e_J$ then its complement $J^c$ is also a maximizer of the same Hamiltonian.
This means that in order to show that the comb strategy (and also the strategy that chooses the second and the third leading expert) is optimal it is sufficient to show that the functions $U_1,...,U_6$ defined by
\begin{align*}
U_1(x):= U(x)-\Phi(x)-\frac{1}{2}\begin{pmatrix}
0\\
0\\
0\\
0
\end{pmatrix}^T \pa^2 U(x)\begin{pmatrix}
0\\
0\\
0\\
0
\end{pmatrix},\\
 U_2(x):= U(x)-\Phi(x)-\frac{1}{2}\begin{pmatrix}
0\\
0\\
0\\
1
\end{pmatrix}^T \pa^2 U(x)\begin{pmatrix}
0\\
0\\
0\\
1
\end{pmatrix},\\
U_3(x):=U(x)-\Phi(x)-\frac{1}{2}\begin{pmatrix}
0\\
0\\
1\\
0
\end{pmatrix}^T \pa^2 U(x)\begin{pmatrix}
0\\
0\\
1\\
0
\end{pmatrix},
\\ 
U_4(x):= U(x)-\Phi(x)-\frac{1}{2}\begin{pmatrix}
0\\
0\\
1\\
1
\end{pmatrix}^T \pa^2 U(x)\begin{pmatrix}
0\\
0\\
1\\
1
\end{pmatrix},\\
U_5(x):= U(x)-\Phi(x)-\frac{1}{2}\begin{pmatrix}
0\\
1\\
0\\
0
\end{pmatrix}^T \pa^2 U(x)\begin{pmatrix}
0\\
1\\
0\\
0
\end{pmatrix},
\\
U_6(x):= U(x)-\Phi(x)-\frac{1}{2}\begin{pmatrix}
0\\
1\\
1\\
1
\end{pmatrix}^T \pa^2 U(x)\begin{pmatrix}
0\\
1\\
1\\
1
\end{pmatrix},
\end{align*}
are non-negative. 
We study each term separately. For the first term we have 
\begin{align*}
U_1(x)= U(x)-\Phi(x)-\frac{1}{2}\begin{pmatrix}
0\\
0\\
0\\
0
\end{pmatrix}^T \pa^2 U(x)\begin{pmatrix}
0\\
0\\
0\\
0
\end{pmatrix}=v(x^{(2)}-x^{(1)},x^{(3)}-x^{(2)},x^{(4)}-x^{(3)})\geq 0
\end{align*}
due to the definition of $v$. 
Additionally we have the following identities for $x\in \cW_4$ that can be computed via Mathematica\footnote{The code of the computation is available in \cite{BEYMath}}. 
\begin{align}
\notag U_4\left(\frac{x}{\sqrt{2}}\right)=&\frac{e^{x_4-x_2}(e^{2x_1}-e^{2x_3})(e^{2x_3}-e^{2x_2})}{2\sqrt{2}(e^{2(x_1+x_2)}-e^{2(x_3+x_4)})}\geq 0,\\
\notag\frac{\sqrt{2}(U_3(\sqrt{2}{x})-U_2(\sqrt{2}{x}))}{\sinh(2(x_3-x_4))}=&1-\arctanh(e^{x_1+x_2-x_3-x_4})\cosh(x_1+x_2-x_3-x_4)\\
\label{eq:u3}&+\arctan(e^{x_1+x_2-x_3-x_4})\sinh(x_1+x_2-x_3-x_4)\\
\notag\frac{\sqrt{2}(U_5(\sqrt{2}{x})-U_3(\sqrt{2}{x}))}{\sinh(2(x_2-x_3))}=&-\arctanh(e^{x_1+x_2-x_3-x_4})\cosh(x_1-x_2-x_3+x_4)\\
\label{eq:u5}&+\arctan(e^{x_1+x_2-x_3-x_4})\sinh(x_1-x_2-x_3+x_4)\\
\notag\frac{\sqrt{2}(U_6(\sqrt{2}{x})-U_3(\sqrt{2}{x}))}{\sinh(2(x_3-x_1))}=&\arctanh(e^{x_1+x_2-x_3-x_4})\cosh(x_1-x_2+x_3-x_4)\\
\label{eq:u6}&+\arctan(e^{x_1+x_2-x_3-x_4})\sinh(x_1-x_2+x_3-x_4).
\end{align}
Due to $x\in \cW_4$, $U_4(x)\geq 0$. 
Additionally, the function 
$$x\geq 0\mapsto 1-\arctanh(e^{-x})\cosh(-x)+\arctan(e^{-x})\sinh(-x)$$
is non-positive.
Thus $$U_3\geq U_2.$$
Finding the sign of the right hand side of \eqref{eq:u5} and \eqref{eq:u6} is equivalent to finding the signs of 
$$-\arctanh(e^{-x})\cosh(-x+y)+\arctan(e^{-x})\sinh(-x+y),\mbox{ for }x,y\geq 0$$
and 
$$\arctanh(e^{-x})\cosh(-x+y)+\arctan(e^{-x})\sinh(-x+y),\mbox{ for }x,y\geq 0.$$
These functions are respectively non-positive and non-negative due to the fact that $\arctanh(e^{-x})\geq \arctan(e^{-x})\geq 0$ and $\cosh(x)\geq |\sinh(x)|$. 
Thus $$U_5\geq U_3\mbox{ and } U_6\geq U_3.$$
Finally, to finish the proof of the main theorem, it is sufficient to show that 
\begin{align}\label{eq:ineku2}U_2\geq 0.\end{align}
To show this inequality, it is more convenient to write $U_2$ as in terms of $v$.
Thanks to \eqref{def:U}, 
$$U_2(x)=v(x_2-x_1,x_3-x_2,x_4-x_3)-\frac{1}{2}\begin{pmatrix}
0\\
0\\
1
\end{pmatrix}^T \pa^2 v(x_2-x_1,x_3-x_2,x_4-x_3)\begin{pmatrix}
0\\
0\\
1
\end{pmatrix},
$$
and to show \eqref{eq:ineku2}, it is sufficient to show that for all $y_1,y_2,y_3\geq 0$, 
$$v_2(y_1,y_2,y_3):=v(y_1,y_2,y_3)-\frac{1}{2}\begin{pmatrix}
0\\
0\\
1
\end{pmatrix}^T \pa^2 v(y_1,y_2,y_3)\begin{pmatrix}
0\\
0\\
1
\end{pmatrix}\geq 0.
$$
Thanks to the smoothness of $v$ on $(0,\infty)^3$ and the fact that the data of \eqref{eq:pdev} is constant, we can differentiate \eqref{eq:pdev} to obtain that 
$v_2$ also solves \eqref{eq:pdev}. Thanks to the maximum principle for this PDE, in order to show \eqref{eq:pdev}, it is sufficient to show that $v_2\geq 0$ for $y_1=0$ or $y_2=0$ or $y_3=0$. Our objective is to use the Proposition \ref{prop:comp}. Similarly to \eqref{eq:deffrh1}-\eqref{eq:deffrh4} define 
\begin{align*}
\tilde f(x,y)&= v_2\left(0,\frac{x}{\sqrt{2}},\frac{y}{\sqrt{2}}\right),\\
\tilde r_1(x,y)&= v_2\left(\frac{x}{\sqrt{2}},0,\frac{y+x}{\sqrt{2}}\right),\\
\tilde h(x,y)&= v_2\left(\frac{y}{\sqrt{2}},\frac{x}{\sqrt{2}},0\right),\\
\tilde r_2(x,y)&= v_2\left(\frac{x+y}{\sqrt{2}},0,\frac{x}{\sqrt{2}}\right).
\end{align*}
By direct computation via Mathematica\footnote{The code of the computation is available in \cite{BEYMath}}, these functions satisfy, 
\begin{align*}
(\pa_x-2\pa_y)\tilde f(x,y)&=\frac{2}{\tanh{x}}\tilde f(x,y)-\frac{2}{\sinh{x}}\tilde r_1(x,y),\\
\pa_x \tilde r_1(x,y)&=-\frac{2}{\sinh{x}} \tilde f(x,y)+\frac{2}{\tanh{x}}\tilde r_1(x,y),\\
(\pa_x-2\pa_y)\tilde h(x,y)&=\frac{2}{\tanh{x}} \tilde h(x,y)-\frac{2}{\sinh{x}} \tilde r_2(x,y)\\
&+\frac{1}{\sqrt{2}}\left(1-\arctanh(e^{-x-y/2})\cosh(x+y/2)-\arctan(e^{-x-y/2})\sinh(x+y/2))\right),\\
\pa_x \tilde r_2(x,y)&=-\frac{2}{\sinh{x}} \tilde h(x,y)+\frac{2}{\tanh{x}} \tilde r_2(x,y).
\end{align*}
Since the function $$x\geq 0 \mapsto 1-\arctanh(e^{-x})\cosh(x)-\arctan(e^{-x})\sinh(x)$$
is non-positive, we have that 
 \begin{align*}
(\pa_x-2\pa_y)\tilde f(x,y)&=\frac{2}{\tanh{x}}\tilde f(x,y)-\frac{2}{\sinh{x}}\tilde r_1(x,y),\\
\pa_x \tilde r_1(x,y)&=-\frac{2}{\sinh{x}} \tilde f(x,y)+\frac{2}{\tanh{x}}\tilde r_1(x,y),\\
(\pa_x-2\pa_y)\tilde h(x,y)&\leq \frac{2}{\tanh{x}} \tilde h(x,y)-\frac{2}{\sinh{x}} \tilde r_2(x,y),\\
\pa_x \tilde r_2(x,y)&=-\frac{2}{\sinh{x}} \tilde h(x,y)+\frac{2}{\tanh{x}} \tilde r_2(x,y).
\end{align*}
Thus, to finish the proof of the main result by application of Proposition \ref{prop:comp}, we need to control 
$\tilde f$ and $\tilde h$ on the boundary of their domain of definition and obtain the limit of $\tilde r_1$ and $\tilde r_2$ at infinity. 
{Note that $\tilde r_1$ and $\tilde r_2$ converge to $0$ at infinity.
By a direct computation\footnote{The code of the computation and the expressions for the functions are available in \cite{BEYMath}.}, we have that 
\begin{align*}
\tilde f(x,y)=&F_f(x,y)+\left(\arctan(e^{-x-\frac{y}{2}})-e^{-x-\frac{y}{2}}+\frac{e^{-3x-\frac{3y}{2}}}{3}\right)H_f(x,y)\\&+\left(\arctanh(e^{-x-\frac{y}{2}})-e^{-x-\frac{y}{2}}-\frac{e^{-3x-\frac{3y}{2}}}{3}\right)G_f(x,y),\\
\tilde h(x,y)=&F_h(x,y)+\left(\arctan(e^{-x-\frac{y}{2}})-e^{-x-\frac{y}{2}}+\frac{e^{-3x-\frac{3y}{2}}}{3}\right)H_h(x,y)\\
&+\left(\arctanh(e^{-x-\frac{y}{2}})-e^{-x-\frac{y}{2}}-\frac{e^{-3x-\frac{3y}{2}}}{3}\right)G_h(x,y),
\end{align*}
where as $x^2+y^2\to \infty$,
\begin{align*}
&F_f(x,y):=\frac{e^{-2(2x+y)}\left(4\cosh(y)+6\csch(2x+y)\sinh^2(x)-\sinh(y)\right)}{24\sqrt{2}}=o(1),\\
&H_f(x,y):=\frac{3\cosh(x-\frac{y}{2})+6\cosh(x+\frac{y}{2})-5\cosh(x+\frac{3y}{2})}{16\sqrt{2}}=o(e^{5x+\frac{5y}{2}}),\\ &G_f(x,y):=\frac{3\sinh(x-\frac{y}{2})-6\sinh(x+\frac{y}{2})-5\sinh(x+\frac{3y}{2})}{16\sqrt{2}}=o(e^{5x+\frac{5y}{2}}),\\
&F_h(x,y):=\frac{e^{-2(2x+y)}\left(4+3\coth(2x+y)-3\cosh(y)\csch(2x+y)\right)}{24\sqrt{2}}=o(1),\\
&H_h(x,y):=\frac{(-1+3\cosh(y))\cosh(x+\frac{y}{2})}{8\sqrt{2}}=o(e^{5x+\frac{5y}{2}}),\\
&G_h(x,y):=\frac{(1+3\cosh(y))\sinh(x+\frac{y}{2})}{8\sqrt{2}}=o(e^{5x+\frac{5y}{2}}).
\end{align*}
Given also the expansions at $0$
$$\arctan(x)=x-\frac{x^3}{3}+O(x^5),\quad \arctanh(x)=x+\frac{x^3}{3}+O(x^5),$$
we have that 
$$\lim_{x^2+y^2\to \infty}\tilde f(x,y)=\lim_{x^2+y^2\to \infty}\tilde h(x,y)=0.$$
} Additionally, 
\begin{align*}
&\tilde f(x,0)=\tilde h(x,0)= v_2(0,x/\sqrt{2},0)\\
&=\frac{1}{8\sqrt{2}}\left(2\arctan(e^{-x})\cosh(x)-4\arctanh(e^{-x})\sinh(x)+\tanh(x)\right),\\
&\tilde f(0,y)= \frac{5}{8 \sqrt{2} }e^y (-1 + \coth(y)) \sinh^2(y)+\frac{1}{16 \sqrt{2} }e^y (-1 + \coth(y)) \sinh(y) \\
&\left(\arctan(e^{-y/2}) (9 \cosh(y/2) - 5 \cosh(3 y/2) - 
   \arctanh(e^{-y/2}) (9 \sinh(y/2) + 5 \sinh(3 y/2))\right),\\
&\tilde h(0,y)= \frac{1}{16 \sqrt{2}}e^y (-1 + \coth(y)) \sinh(
  y) \\
  &\left(\arctan(e^{-y/2}) (\cosh(y/2) + 3 \cosh(3 y/2)) + 
   \arctanh(e^{-y/2}) (\sinh(y/2) - 3 \sinh(3 y/2))\right).
\end{align*}
These functions are all non-negative. 
Direct application of Proposition \ref{prop:comp} then yields 
$$\tilde f, \tilde h, \tilde r_1,\tilde r_2\geq 0\mbox{ on }[0,\infty)^2.$$
Thus, for all $x\in \cW_4$ we have
$$U(x)-\frac{1}{2}\sup_{J\in P(N)} e_J^\top \pa^2 U(x)e_J= \Phi(x).$$
Thanks to the smoothness and symmetry of $u$, we obtain \eqref{eq:pdegeo}.

\subsection{Proof of Theorem \ref{thm:asyopt}}\label{sec:profossrst}
We first prove the asymptotics for $\underline u^\d$. This function satisfies the dynamic programming principle
\begin{align*}
\underline  u^\d(x)&=\delta \Phi(x)
\\&+\frac{1-\d}{2}\inf_{\a\in\cU} \left(\underline  u^\d(x+\sqrt{\d}e_{\cJ_\cC(x)})-\a(\cJ_\cC(x))+\underline  u^\d(x+\sqrt{\d}e_{\cJ_\cC^c(x)})-\a(\cJ_\cC^c(x))\right)\\
&=\delta \Phi(x)+\frac{1-\d}{2} \left(\underline  u^\d(x+\sqrt{\d}e_{\cJ_\cC(x)})+\underline u^\d(x+\sqrt{\d}e_{\cJ_\cC^c(x)})-1\right)
\end{align*}
This is equivalent to 
\begin{align*}
\underline u^\d(x)= \Phi(x)+\frac{1-\d}{2\d} \left(\underline  u^\d(x+\sqrt{\d}e_{\cJ_\cC(x)})+\underline u^\d(x+\sqrt{\d}e_{\cJ_\cC^c(x)})-1-2\underline  u^\d(x)\right).
\end{align*}
Similarly to $u$ and $V^\d$, 
$$\underline  u^\d(x+\sqrt{\d}\lambda \sum_{i=1}^4 e_i)=\underline  u^\d(x)+\lambda\mbox{ for all }\lambda \in \R .$$
Thus, 
$$\underline u^\d(x+\sqrt{\d}e_{\cJ_\cC^c(x)})-1=\underline  u^\d(x+\sqrt{\d}e_{\cJ_\cC^c(x)}-\sqrt{\d}\sum_{i=1}^4 e_i)=\underline  u^\d(x-\sqrt{\d}e_{\cJ_\cC(x)})$$ and the DPP becomes 
\begin{align*}
\underline u^\d(x)= \Phi(x)+\frac{1-\d}{2\d} \left(\underline u^\d(x+\sqrt{\d}e_{\cJ_\cC(x)})+\underline  u^\d(x-\sqrt{\d}e_{\cJ_\cC(x)})-2\underline  u^\d(x)\right).
\end{align*}
Due the fact that $\cJ^b_\cC$ is balanced, $\underline u^\d$ in fact does not depend on $  \a\in \cU$. Thus, by choosing a particular control we can prove similarly to the proof of \cite[Theorem 7]{MR3768426} that  
$\underline  u^\d$ converges to the unique viscosity solution of the equation 
$$f(x)-\frac{1}{2} e_{\cJ_\cC(x)}^\top \pa^2 f(x)e_{\cJ_\cC(x)}= \Phi(x)$$
with linear growth. 
Note that thanks to \eqref{eq:PDEUU}, $u$ also solves this PDE and has linear growth. Thus, comb strategies are asymptotically optimal and $\underline  u^\d(x)\to u(x)$ as $\d\downarrow 0$. 
\section{Concluding Remarks}
Using a system of first order hyperbolic PDE, \eqref{syst:hf}-\eqref{syst:hr}, we characterize and compute the expectation \eqref{eq:defv} of the third component of the local time of an obliquely reflected Brownian motion in the first octant. Then, using a maximum principle in Proposition \ref{prop:comp}, we show that this value provides a solution to the Hamilton-Jacobi-Bellman equation \eqref{eq:pdegeo} that characterizes the long time behavior of a regret minimization problem with $4$ experts. Finally, we prove that, as conjectured in \cite{MR3478415}, comb strategies are asymptotically optimal for the nature.

We conjecture that this methodology can be performed for $N\geq 5 $ experts. The starting point of our computation is the Proposition \ref{prop:diag} where we compute the value in an invariant set for the flow of the obliquely reflected Brownian motion. Similarly, to follow such a methodology for $N\geq 5$, one needs to compute the invariant sets of the obliquely reflected Brownian motion in the positive orthant of dimension $N-1$ and compute the expectations in these sets. Since they are consequence of the flow property for the reflected Brownian, one can expect that for $N\geq 5$, a first order system similar to \eqref{syst:hf}-\eqref{syst:hr} can be established to compute the value on the faces of the first orthant using the value in the invariant sets of the obliquely reflected Brownian motion. 

We also mention that one can use our methodology to study the parabolic version of \eqref{eq:pdegeo} which corresponds to the long time behavior of the game with deterministic stopping. In this case, we expect that the system \eqref{syst:hf}-\eqref{syst:hr} has an additional time dependence. 

\bibliography{ref}{}
\bibliographystyle{plain}

\end{document}